\documentclass[11pt]{article}
\usepackage{latexsym,amsfonts,amssymb,amsmath,amsthm}
\usepackage{graphicx}
\usepackage{caption}
\usepackage{subcaption}

\usepackage[usenames,dvipsnames]{color}
\usepackage{ulem,comment}

\usepackage{color,cite}
\usepackage{url}
\usepackage{array}

\usepackage{relsize}

\begin{document}

\newtheorem{tm}{Theorem}[section]
\newtheorem{prop}[tm]{Proposition}
\newtheorem{defin}[tm]{Definition} 
\newtheorem{coro}[tm]{Corollary}
\newtheorem{lem}[tm]{Lemma}
\newtheorem{assumption}[tm]{Assumption}
\newtheorem{rk}[tm]{Remark}
\newtheorem{nota}[tm]{Notation}
\numberwithin{equation}{section}

\newcommand{\stk}[2]{\stackrel{#1}{#2}}
\newcommand{\dwn}[1]{{\scriptstyle #1}\downarrow}
\newcommand{\upa}[1]{{\scriptstyle #1}\uparrow}
\newcommand{\nea}[1]{{\scriptstyle #1}\nearrow}
\newcommand{\sea}[1]{\searrow {\scriptstyle #1}}
\newcommand{\csti}[3]{(#1+1) (#2)^{1/ (#1+1)} (#1)^{- #1
 / (#1+1)} (#3)^{ #1 / (#1 +1)}}
\newcommand{\RR}[1]{\mathbb{#1}}

\newcommand{\rd}{{\mathbb R^d}}
\newcommand{\ep}{\varepsilon}
\newcommand{\rr}{{\mathbb R}}
\newcommand{\alert}[1]{\fbox{#1}}
\newcommand{\eqd}{\sim}
\def\p{\partial}
\def\R{{\mathbb R}}
\def\N{{\mathbb N}}
\def\Q{{\mathbb Q}}
\def\C{{\mathbb C}}
\def\l{{\langle}}
\def\r{\rangle}
\def\t{\tau}
\def\k{\kappa}
\def\a{\alpha}
\def\la{\lambda}
\def\De{\Delta}
\def\de{\delta}
\def\ga{\gamma}
\def\Ga{\Gamma}
\def\ep{\varepsilon}
\def\eps{\varepsilon}
\def\si{\sigma}
\def\Re {{\rm Re}\,}
\def\Im {{\rm Im}\,}
\def\E{{\mathbb E}}
\def\P{{\mathbb P}}
\def\Z{{\mathbb Z}}
\def\D{{\mathbb D}}
\newcommand{\ceil}[1]{\lceil{#1}\rceil}

\title{Analysis of a time-delayed HIV/AIDS epidemic model with education campaigns.
}

\author{
Dawit Denu\footnote{Department of Mathematics, Georgia Southern University, Savannah, GA 31419 (ddenu@georgiasouthern.edu)}\qquad   \quad \quad 
  Sedar Ngoma\footnote{Department of Mathematics, State University of New York at Geneseo, NY 14454 (ngoma@geneseo.edu)} \qquad \quad \quad
 Rachidi B. Salako\footnote{Department of Mathematics, University of Nevada at Las Vegas, Las Vegas, NV 89154(rachidi.salako@unlv.edu)}  }

\date{}
\maketitle

\begin{abstract}
We consider a time-delayed HIV/AIDS epidemic model  with education dissemination and study the asymptotic dynamics of solutions as well as the asymptotic behavior of the endemic equilibrium with respect to  the amount of information disseminated about the disease. Under appropriate assumptions on the infection rates, we show that if the basic reproduction number is less than or equal to one, then the disease will be eradicated in the long run and any solution to the Cauchy problem converges to the unique disease-free equilibrium of the model.  On the other hand, when the basic reproduction number is greater than one, we prove that the disease will be permanent but its impact on the population can be significantly minimized as the amount of education dissemination increases. In particular, under appropriate hypothesis on the model parameters, we establish that the size of the component of the infected population of the endemic equilibrium decreases linearly as a function of the amount of information disseminated. 
We also fit our model to a set of data on HIV/AIDS in order to estimate the infection, effective response, and information rates of the disease. We then use these estimates to present numerical simulations to illustrate our theoretical findings.
\end{abstract}




\section{Introduction}
The human immunodeficiency virus (HIV) is the virus which is responsible for causing acquired immunodeficiency syndrome (AIDS). The virus cannot replicate by its own and thus in order to reproduce it will hijack cells (CD4) that play a central role in responding to the invasion of infections in the body. The virus  can spread through some body fluids and once it is inside the body it can destroy the immune system, making the person too weak to fight other infections. The elevated viral load of someone recently infected with HIV is the main biological reason that they are more likely to transmit HIV to others. The higher the viral load, the greater the risk is of transmitting HIV to others. When the virus has destroyed a certain number of CD4 cells and the CD4 count drops below 200 cells per $mm^3$, a person will have progressed to AIDS. Until now there is no vaccine to prevent HIV infection nor a cure for AIDS and if HIV is untreated earlier it will lead to AIDS in about 8 to 10 years. However, there are some medications and antiviral treatments to slow  down the progression of the disease \cite{ragni1992effect}. According to UNAIDS, in 2018 there were approximately 38 million people with HIV/AIDS worldwide. Out of these many people the majority ($95.5\%$) were adults and the remaining were children of age less than 15 years old. The use and access to antiretroviral therapy (ART) has been gradually improved and  in 2018 close to 23 million people with HIV were using this treatment. Despite the fact that HIV/AIDS is not one of the leading causes of death in developed countries due to efficient HIV prevention and treatment techniques, it remains a leading cause of death for certain age groups and people living in developing countries. At the international AIDS conference in Durban, South Africa, in July 2016, the world health organization (WHO) reported four key challenges caused by the HIV/AIDS pandemic  \cite{friedland2016marking,hakim2016oral}. These are the need to renew attention to HIV prevention, scaling up access to HIV treatments, the growing emergence of antiretroviral drug resistance, and the need for sustainable financing of the global response.

In the absence of effective medical procedures to prevent or cure HIV infection and AIDS, experts have repeatedly emphasized education as the main means to prevent the spread of the infection. These may include dissemination of information on how the infection is transmitted from one person to another, fostering social support, helping people to change certain risky behaviors such as unsafe sex, sharing needles or syringes and others \cite{walque2004does,bhunu2011mathematical,nyabadza2006mathematical,nyabadza2010analysis}. In addition to the aforementioned general strategies, different countries have been implementing various techniques to reduce  behaviors  that make individuals more vulnerable to becoming infected, or infecting others, with HIV. These interventions have generally aimed to increase the use of condoms, decrease the number of sexual partners, delay the start of sexual intercourse, encourage young persons to practice sexual abstinence, decrease sharing of needles, syringes, and substance use \cite{de2006gets,des1992international,kirby2007sex}. 

In order to understand the impact of preventive control strategies on the spread of HIV infection, various authors have been studying a number of mathematical models and showed that the different interventions mentioned above are very important to prevent a rampant HIV/AIDS transmission \cite{joshietal2008, velasco1994modelling, mukandavire2009modelling, mukandavire2007asymptotic}. Education and media campaign have a significant effect in changing the behavior of susceptible people by raising their awareness about infections such as HIV. In this regards,  Joshi et al. \cite{joshietal2008} proposed a HIV/AIDS epidemic model to study the effect of information campaign on the HIV epidemic in Uganda. Their simulation results predict that an adequate information campaign can successfully  reduce both prevalence and incidence of HIV infection. They considered SIRE epidemic model, where the first three letters represent the susceptible, infected and removed classes, respectively. The compartment E represents the information and education campaigns (IEC), which upon interaction with the susceptible group will change the behavior of the susceptible individual.   The authors have included two types of information campaigns, collectively known as ABC, that were practiced in Uganda to mitigate the power of HIV infection since 1992 \cite{joshietal2008}. The first strategy, abstinence and be faithful campaign (the AB behavior), was mainly changing people behavior based on a combination of ``risk avoidance" and harm reduction by encouraging young adults to delay sexual debut or eliminate casual or other concurrent sex partners to reduce exposure to HIV. The second option was to encourage people to practice safe-sex by the use of condom (the C behavior). As a result when the susceptible group S interacts with the information, then part of the group will change its behavior and practice the abstinence and be faithful behaviors to join the $S_{AB}$ group. Similarly, the interaction of the susceptible group with the information will result in changing the behavior of some of the people to use condom and as a result they will join the $S_C$ group. 

  One way to better capture the dynamics of HIV starting from the time of cell's infection to the time it produces new virions is to introduce a latency (or delay) time $u>0$. Time delay can arise for various practical reasons in epidemiology. For example,  pharmacological delay occurs between the ingestion of drug and its appearance within cells, while intracellular delay is the time between initial infection of a cell by HIV and the release of new virions. 
In this perspective, there are several HIV epidemic models that have been developed to study the effect of the delays on the overall dynamics of the infection \cite{culshaw2000delay, lutambi2016effect, mukandavire2007asymptotic, roy2013long, cai2009stability, naresh2011nonlinear}. In particular,  Z. Mukandavire et al. presented a mathematical model for HIV/AIDS with explicit incubation period as a system of discrete time delay differential equations \cite{mukandavire2007asymptotic}. They analyzed the stability of equilibrium points using the Lyapunov functional approach. Moreover, they investigated the positivity and boundedness of the solutions, and provided conditions for the permanence of the infection. 

Thus, motivated by the above discussion, due to the advantage of the positive effect of information and education campaigns to control the transmission of HIV/AIDS and the fact that delay models provide a more realistic dynamic of the infection, we propose and analyze a time-delayed HIV/AIDS model with the presence of information and education campaigns.

\subsection{Model formulation}
In this paper we consider a population of total size $N(t)$ at time $t$, with a constant recruitment rate of $B$. The population size $N(t)$ is divided into five compartments $S_0(t)$, $S_1(t)$,  $S_2(t)$, $I(t)$, and $R(t)$.  The $S_0$-class, also called the general susceptible class, denotes the sub-group of the population which are not yet infected but are susceptible to contract the disease upon interaction with infected people. Furthermore, the $S_0$-class individuals have direct access to information and education campaigns about the disease, which includes practicing abstinence and faithfulness and the use of condom. 
We denote by $Z(t)$ the amount of the educational information from the information and education campaigns (IEC).
 Thanks to the access to education campaigns, part of the $S_0$ population will change their behavior and join other compartments according to the type of information they receive and choose to practice.
The group $S_1(t)$ denotes those individuals who change their behavior and practice abstinence and faithfulness as a result of the information and education campaigns. Similarly, $S_2(t)$ represents the group that changes their behavior and started  practicing safe sex with the use of condoms. 
There are some individuals who are not only exposed to the information and education campaigns, but also used the provided information to the fullest that they will never get infected by HIV throughout their life. It is also known that a small proportion of humans show apparently complete inborn resistance to HIV \cite{Susan,Sarah}. The main mechanism is a mutation of the gene encoding CCR5, which acts as a co-receptor for HIV. These individuals are grouped under the removed compartment $R(t)$. Note that we assume that it is through natural death and infection that people can leave the $S_i$, $i=1,2$, groups. We also assume that it is only through natural death that people can leave the $R$-group.
We assume that the interaction between $S_0$ and $Z$ will have rates of $\gamma_0$ to go to the $R$ group, $\gamma_1$ to go to the $S_1$ group, and $\gamma_2$ to go to the $S_2$ class. 

The density of information $Z(t)$ is assumed to be proportional to the number of regional and global organizations providing information and educational campaigns on HIV prevention, and the number of such organizations is in general proportional to the number of infections in the population.
From a more realistic perspective, one of the powerful ways to combat infectious diseases is to make sure that the population is well informed about the transmission mechanisms of infectious diseases. To accomplish this task, health authorities provide information and education campaigns within a given population of interest. Hence the inclusion of the information and education campaigns,  $Z(t)$, in our model. These campaigns are so important because they help spread a word throughout the population, change individuals behaviors for the best, and prevent the population from contracting the disease. For the case of HIV/AIDS, the World Health Organization (WHO) and the Center for Disease Control and Prevention (CDC) both have outlined important intervention strategies to use in order to protect against HIV/AIDS. These include but are not limited to abstinence (not having sex), never share needles, using condom the right way every time one has sex, being faithful, and the use of antiretroviral drugs (ARVs) for prevention such as pre-exposure prophylaxis or PrEP, which is medicine people at risk for HIV take to prevent getting HIV from sex or injection drug use, and post-exposure prophylaxis or PEP, which is medicine to take to prevent HIV after a possible exposure. Motivated by this short description, our model assumed the population is exposed to information and education campaigns and consequently changes its behavior according to either practicing abstinence and being faithful or using condom  \cite{belle2019behavioral, margevicius2013influence, Fan1993QuantitativeEF}.

We suppose that there is intracellular constant time delay $u\ge 0$ between the initial infection of cells by HIV and when a cell becomes infectious. Time delay is very important especially for an application point of view.  Indeed, according to the CDC,  getting and keeping a low undetectable viral load prevents HIV transmission during sex. Antiretroviral therapy can help  keep the viral load very low, called viral suppression;  and thereby creating some delay in the time of being infectious for newly infected HIV patients. This time delay is vital because it could lead to lower the basic reproduction number of the disease and consequently reduces significantly its impact on the population. To illustrate, in Section~\ref{numerics}, the simulations indicate that as the time delay increases, the infection rate decreases. This is an indication to suggest that if we are able to find ways to increase the intracellular time delay between the initial infection of cells by HIV and when those cells become infectious, we would be able to prevent many HIV infections.   

The infection transmission rates upon interaction with the infected group $I(t)$ are denoted by $\beta_{0,u},\beta_{1,u}$, and $\beta_{2,u}$  and correspond respectively, to  the groups $S_0, S_1$, and $S_2$. 
In addition to the  natural death rate $\mu$, which is the same for all subgroups, we denote  by $d>0$  the HIV/AIDS induced death rate on the infected population.  A schematic relationship between model classes is provided in Figure~\ref{fig:flowchart}.
\begin{figure}[ht]
\centering
\includegraphics[width=.8\textwidth,height=.3\textheight]{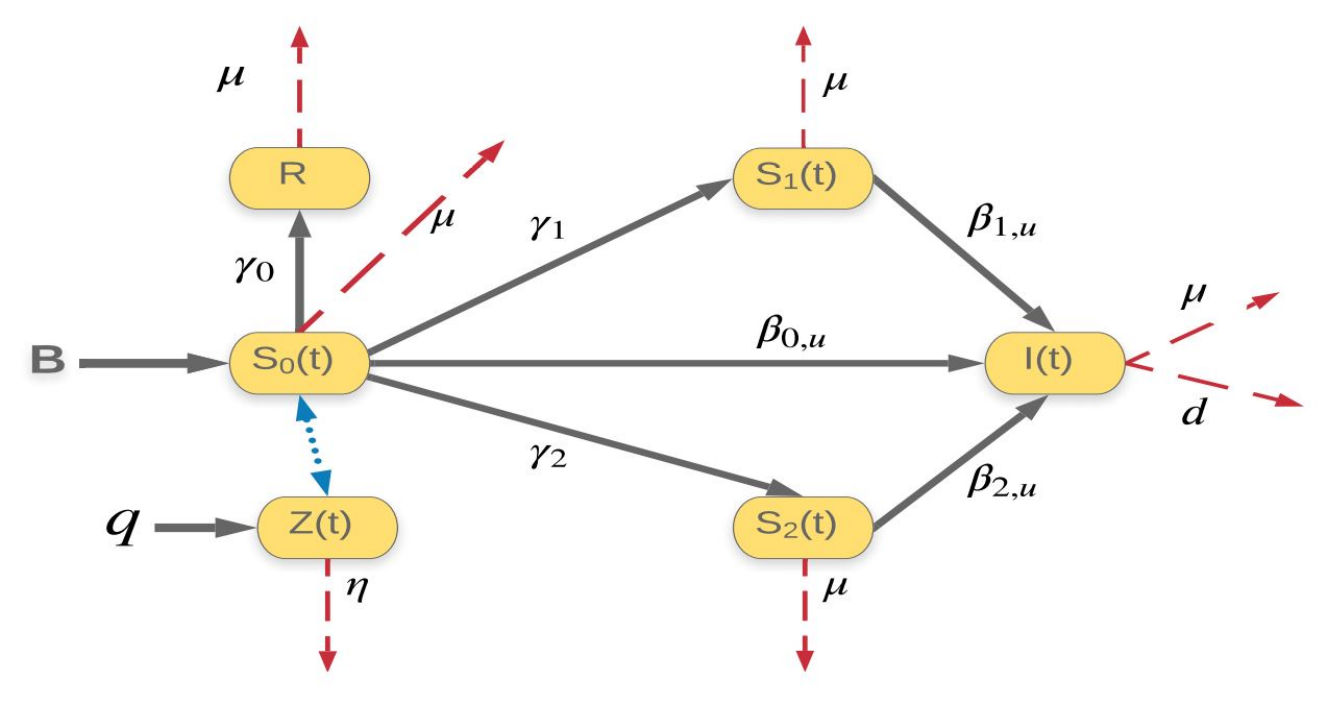}
\caption{A flow diagram describing the dynamics of HIV/AIDS epidemic model where $S_0(t), S_1(t), \text{ and }S_2(t)$ are the different susceptible groups of the human population, $I(t) \text{ and } R(t)$ are the number of infected and removed groups, respectively. Also, $Z(t)$ is the information density at time $t$, and the broken blue arrow indicates the interaction of the $S_0(t)$ group with the information $Z(t)$ to produce $S_1(t)$, $S_2(t)$, and $R(t)$. The solid arrows represent transitions between the different epidemiological classes, whereas the red broken arrows represent the death rates of the various groups in the model. }
\label{fig:flowchart}
\end{figure}

Table~\ref{tab:table1} provides a summary description of the model parameters. The values of the parameters in Table~\ref{tab:table1} together with the parameters to be estimated are provided in Section~\ref{sec3}.
\begin{table}[htbp]
 \caption{A list of model parameters and their interpretations}
    \centering
    \begin{tabular}{|c||c|}\hline
    Symbols  & Description \\\hline
 $u$ & Intracellular time delay \cr
    $\gamma_0, \gamma_1, \gamma_2$ & Effective response rates of $S_0(t), S_1(t), S_2(t) $ \\
    $\beta_{0,u}, \beta_{1,u}, \beta_{2,u}$    & Infection transmission rates of $S_0(t), S_1(t), S_2(t) $ \\
     $\mu$ & Natural death rate of $S_0(t), S_1(t), S_2(t), I(t), R(t)$ \\
    $d$ & Death rate due to the disease \\
    $q$ & Rate of increase of information w.r.t. $I(t)$ \\
    $\eta$ & Death rate of information\\
    $B$ &  Total rate of entry into  $S_0(t)$
    \\\hline
    \end{tabular}
    \label{tab:table1}
\end{table}

Taking into account the above considerations, the model dynamics is assumed to be governed by the following system of nonlinear 
delay differential equations
\begin{equation}\label{e-0}
\begin{cases}
\frac{d}{dt} S_{0}(t)=  B-\sum_{j=0}^{2} \gamma_{j} S_0(t) Z(t) -\beta_{0,u} S_{0}(t) I(t-u)  -\mu S_{0}(t),\cr
\frac{d}{dt}  S_{j}(t)= \gamma_{j} S_{0}(t) Z(t)-\beta_{j,u} S_{j}(t) I(t-u) -\mu S_{j}(t), \quad j =1, 2, \cr
\frac{d}{dt} Z(t)=q I(t)- \eta Z(t),\cr
\frac{d}{dt} I(t)=\sum_{j=0}^{2} \beta_{j,u} S_{j}(t)I(t-u)  -\mu I(t)-d I(t), \cr
\frac{d}{dt}R (t) = \gamma_0 S_0(t) Z(t) - \mu R(t).
\end{cases}
\end{equation}
 
A brief description of this model follows. We assume that  the  total  rate of entry into the general susceptible group $S_0(t)$ is $B$. The interactions of general susceptible individuals in class $S_0$ with information $Z$ will result in some $S_0$ individuals to move to the $R$ compartment with rate of movement of $\gamma_0$. Some part of the $S_0$ individuals will go from $S_0$ group to the $S_1$ and $S_2$ compartments with rates of movement of $\gamma_1$ and $\gamma_2$, respectively. Moreover, the interactions of $S_0$, $S_1$, and $S_2$ individuals with infectious individuals $I$ will result in those individuals to contract the disease and move to the infectious compartment $I$ with transmission rates $\beta_{0,u}$, $\beta_{1,u}$, and $\beta_{2,u}$, respectively. Furthermore, individuals in all compartment are subject to natural death with the same rate of $\mu$. The infectious individuals can die due to the disease at rate $d$. Finally, for the $Z$ compartment equation, we assume that the change in information within a population is proportional to the number of organizations providing information and education campaigns. The negative term $-\eta$ accounts for the lack of information due to limitation in resources. In fact, population areas with more infections should be subject to more information and education campaigns, and areas with less infection would be provided a corresponding level of information and education campaigns \cite{belle2019behavioral, margevicius2013influence, Fan1993QuantitativeEF}. This justifies our assumption that the growth rate $q$ of education campaign is directly proportional to the number of infected population.   
 
Furthermore, the nonlinear terms in the model such as $\beta_{j,u} S_{j}(t) I(t-u), \quad j= 0, 1, 2$ are obtained using the simple mass action law that assumes the infection transmission rates increases linearly with the population size. Since $\beta_{j,u}, \, j= 0, 1, 2$ represents the infection transmission rates of the $S_j$ compartment, then $\beta_{j,u} S_{j}(t)$ is the total number of contacts of the susceptible group with one infective per unit time. Thus $\beta_{j,u} S_{j}(t) I(t-u)$ will be the total number of contacts of of susceptibles with the
entire infective class per unit time. If needed the simple mass action used in this model can be easily modified to other incidence types such as standard, saturated or nonlinear incidence rates.  
 It is worth mentioning that the standard incidence rate is more suitable to model sexually transmitted infectious diseases. In this work, we use the mass action incidence in our model because it provides a good fit between data and model simulations (see section \ref{sec3}).   
 
It is the aim of the current work to study the dynamics of solutions of \eqref{e-0}. We also present numerical simulations to illustrate our theoretical findings.

\subsection{Statement of main results}
In this subsection we state our main results. We observe that the  first five equations of \eqref{e-0} decouple from the last equation. Furthermore, we note that the solution to the last equation in \eqref{e-0} is completely determined by the solution of the decouple sub-system formed by the first five equations. For this reason, we shall focus on the  study of the dynamics of the sub-system
\begin{align}\label{e-0-1}
\begin{cases}
\frac{d}{dt} S_{0}(t)  =  B-\sum_{j=0}^{2} \gamma_{j} S_0(t) Z(t) -\beta_{0,u} S_{0}(t) I(t-u)  -\mu S_{0}(t),\cr
\frac{d}{dt}  S_{j}(t)= \gamma_{j} S_{0}(t) Z(t)-\beta_{j,u} S_{j}(t) I(t-u) -\mu S_{j}(t), \quad j =1, 2,\cr
\frac{d}{dt} Z(t) = q I(t)- \eta Z(t),\cr 
\frac{d}{dt} I(t)=\sum_{j=0}^{2} \beta_{j,u} S_{j}(t)I(t-u)  -\mu I(t)-d I(t),  \cr
\end{cases}
\end{align}
where $q,\eta,\mu,d,B>0$, $\beta_{j,u}>0$ for every $j\in\{0,1,2\}$ and $\gamma_j>0$ for every $j\in\{0,1,2\}$, and  subject to the initial conditions 
\begin{equation}\label{initial-cond}
    (S_0(0),S_1(0),S_2(0),Z(0))\in \mathbb{R}_{+}^{4},\ S_0(0)>0,\ I_0(\cdot)\in C([-u,0]:\mathbb{R}_+),
\end{equation}
where $\mathbb{R}_+=\{x\in\mathbb{R}\,:\,x\ge0\}$. Recall that $u\ge 0$ is the time delay. We endow $ C([-u,0]:\mathbb{R}_+)$ with the sup-norm $\|\phi\|_{\infty}=\sup_{-u\leq s\le 0}|\phi(s)|$. Due to biological interpretation, we suppose that the infections rates $\beta_{j,u}$ is a non-increasing function of $u$ for every $j=0,1,2$. 
By standard theory on Delay-Differential Equations,  \eqref{e-0-1} subject to initial conditions satisfying \eqref{initial-cond} has a unique non-negative solution, which is defined for all $t\ge 0$ (see Section \ref{sec0}). Moreover, with the notations $ S(t):=\sum_{j=0}^2S_j(t)$ and $N(t):=S(t)+I(t)$, it holds that 
$$ 
\frac{d}{dt}N(t)= B-\mu N(t)-dI(t)-\gamma_0Z(t)S_0(t)\leq B-\mu N(t), \ t>0.
$$
Hence, it follows from comparison principle for ODEs that 
\begin{equation}\label{x-0}
\limsup_{t\to\infty}N(t)\le \frac{B}{\mu}\ \text{and}\     N(t)\leq \max\ \left\{N(0),\frac{B}{\mu}\ \right\}, \forall \ t\ge 0.
\end{equation}
Thus, using again the comparison principle for ODEs, we get 
\begin{equation}\label{x0-1}
\limsup_{t\to\infty}Z(t)\leq \frac{q}{\eta}\limsup_{t\to\infty}I(t)\leq \frac{qB}{\eta\mu} \quad \text{and}\quad Z(t)\leq \max\ \left\{Z(0), \frac{q}{\eta}\max\ \left\{N(0),\frac{B}{\mu}\ \right\} \right\}.
\end{equation} These show that the set  
$$
\mathcal{X}_u=\left\{ (S_0,S_1,S_2,Z,I(\cdot))\in R^4_+\times C([-u,0]:\mathbb{R}_+)\ : \ \sum_{j=0}^2S_j+I(0)\leq \frac{B}{\mu} \ \text{and}\ Z\le \frac{qB}{\mu\eta}  \right\}
$$ is forward invariant for the flow generated by solutions of \eqref{e-0-1}. Observe that when $I(s)\equiv 0$  on $[-u,0]$ then $I(t)\equiv 0$ for all $t\ge 0$ and $Z(t)=Z(0)e^{-\eta t}\to 0$ as $t\to\infty$. In which case it is also easily seen that $(S_0(t),S_1(t),S_2(t))\to (\frac{B}{\mu},0,0)$ as $t\to\infty$.  Hence, due to biological interpretations, we shall always suppose that our initial conditions belong to $\mathcal{X}_u$ with the additional properties that $I(s)>0$ for every $s\in[-u,0]$ and $S_0(0)>0$. 

It is easily seen that ${\bf E^0} = \left(\frac{B}{\mu},0, 0, 0, 0 \right)^T$  is always an equilibrium solution of \eqref{e-0-1}. The equilibrium state ${\bf E^0}$ corresponds to the disease-free equilibrium.  
 We note that any equilibrium solution ${\bf \tilde{E}}=(\tilde{S}_0,\tilde{S}_1,\tilde{S}_2,\tilde{Z},\tilde{I})^T$  of \eqref{e-0-1} is uniquely determined by its $\tilde{I}$ coordinate as the remaining components are given by the formulas
\begin{equation}\label{equilibrium-def}
    \tilde{S}_0=\frac{B}{\mu+b_{0,u}\tilde{I}},\quad \tilde{Z}=\frac{q\tilde{I}}{\eta}, \quad 
    \text{and}\quad \tilde{S}_j=\frac{q\gamma_j\tilde{I}B}{\eta(\mu+\beta_{j,u}\tilde{I})(\mu+b_{0,u}\tilde{I})},\ j=1,2,
\end{equation}
where $\gamma:=\sum_{j=0}^2\gamma_j$ and  $b_{0,u}= \beta_{0,u}+\frac{q}{\eta}\gamma$. Any equilibrium solution ${\bf \tilde{E}}=(\tilde{S}_0,\tilde{S}_1,\tilde{S}_2,\tilde{Z},\tilde{I})^T$  of \eqref{e-0-1} for which $\tilde{I}>0$ will be called an endemic equilibrium.

We note that when \eqref{e-0-1} is linearized at the disease free equilibrium ${\bf E^0}$, $I(t)$ satisfies 
$$ 
I_t=\beta_{0,u}\frac{B}{\mu}I(t-u)-(\mu+d)I.
$$
The coefficient $\frac{\beta_{0,u}B}{\mu}$ is then the transition rate of new infected individual while $\mu+d$ is the transition rate from the infected group at the disease-free equilibrium.  Thus, using the next generation matrix method \cite{Diek}, the basic reproduction number is given by 
\begin{equation}\label{R0}
\mathcal{R}_{0,u} = \frac{B\beta_{0,u}}{\mu(\mu + d)}.
\end{equation}
It is well known that the existence of equilibrium solutions and their stability are closely related to the  value of the basic reproduction number $\mathcal{R}_{0,u}$.

To simplify notations in our proofs, we introduce the following new quantity 
\begin{equation}
    \mathcal{D}_0=\frac{\mu+d}{B/\mu},
\end{equation}
which is the ratio of the death rate $\mu+d$ of the infected class (i.e., the rate at which an individual leaves the infective class) by the total size of the population $\frac{B}{\mu}$ when at the disease-free equilibrium. In other words, $\mathcal{D}_0$ represents the per capita death rate of an infective individual in an entirely susceptible population. We will refer to $\mathcal{D}_0$ as the {\it critical death rate} of the model ~\eqref{e-0}.  Observe that $\mathcal{R}_{0,u}\le 1$ if and only if $\beta_{0,u}\le\mathcal{D}_0$ and $\mathcal{R}_{0,u}> 1$ if and only if $\beta_{0,u}>\mathcal{D}_0$.

For this reason, we will use $\mathcal{D}_0$ in the remainder of this paper. Theorem \ref{tm-stability-of-DFE} and Theorem \ref{tm-existence-of-EE} below are good illustrations that justify the introduction of the {\it critical death rate} $\mathcal{D}_0$.

\medskip

\begin{tm}[Stability of the disease-free equilibrium ${\bf E^0}$]\label{tm-stability-of-DFE}
For every time delay $u\ge 0$, the disease-free equilibrium ${\bf E^0}$ is linearly stable if and only if $\beta_{0,u}<\mathcal{D}_0$. Furthermore if $\max\{\beta_{0,u},\beta_{1,u},\beta_{2,u}\}\le \mathcal{D}_0$ then the disease-free equilibrium is globally stable.
\end{tm}

By Theorem \ref{tm-stability-of-DFE}, there is no endemic equilibrium whenever  $\beta_{0,u}=\max\{\beta_{0,u},\beta_{1,u},\beta_{2,u}\}\le\mathcal{D}_0$. In reality, it is expected that when people start adopting some sort of protection behaviors to avoid contracting the disease, then the infection rate decreases as well. Then it is natural to suppose that $\beta_{0,u}>\max\{\beta_{1,u},\beta_{2,u}\}$.  However by curiosity, one might want to know from a mathematical point of view what happens if $\beta_{0,u}<\max\{\beta_{1,u},\beta_{2,u}\}$.   As we shall see from Theorem \ref{lem-1-0-0-0} $(ii)$ later, there is a range of parameters satisfying $\beta_{0,u}\leq \mathcal{D}_0<\max\{\beta_{1,u},\beta_{2,u}\}$ for which ${\bf E}^0$ is not asymptotically stable. In such case there are two endemic equilibrium solutions of \eqref{e-0-1}. 

Our next result is about the existence and uniqueness of the endemic equilibrium when  $\beta_{0,u}>\mathcal{D}_0$.

\medskip

\begin{tm}[Existence and uniqueness of endemic equilibrium point]\label{tm-existence-of-EE}
System \eqref{e-0-1} has a unique endemic equilibrium solution whenever $\beta_{0,u}>\mathcal{D}_0$.  

\end{tm}




We first remark that Theorem \ref{tm-existence-of-EE} follows from Theorem \ref{lem-1-0-0-0} below. Theorem \ref{lem-1-0-0-0}, which is of independent interest, provides a complete picture of the existence of endemic solutions of \eqref{e-0-1}.  Suppose  $\beta_{0,u}>\mathcal{D}_0$ and let ${\bf E^*}=(S_0^*,S_1^*,S_2^*,Z^*,I^*)^T$ denotes the unique endemic equilibrium solution of \eqref{e-0-1} where $S_0^*$, $S_1^*$, $S_2^*$ and  $Z^*$ are given by \eqref{equilibrium-def} with $\tilde{I}=I^*$. The existence of $I^*$ is uniquely determined by the unique positive solution of the algebraic equation
\begin{equation}\label{I-star}
    G(I)=\mu+d=\frac{B\mathcal{D}_0}{\mu},
\end{equation} 
where \begin{equation}\label{g-defi-1}
G(I):=\frac{\beta_{0,u}B}{\mu +b_{0,u}I} +\frac{Bq}{\eta}\sum_{j=1}^2\frac{\beta_{j,u}\gamma_jI}{(\mu +\beta_{j,u}I)(\mu+b_{0,u}I)}.
\end{equation}
 It is clear that an explicit expression is very difficult to derive for $I^*$. This makes the analysis of the endemic equilibrium solution more difficult. In fact, the study of the sign of the real parts of the roots of the characteristic equation of the linearized system at ${\bf E}^*$, see \eqref{p-0}, is more complex due to its form. 
 A plot of the nonlinear function $G(I)$ and the line $y = \mu + d$ together with the value of $I^*$, their intersection point, is given in Figure~\ref{fig:Istar}. Note that this graph has been produced with parameter values estimated in Section~\ref{sec3}.
\begin{figure}[htpb]
    \centering
    \includegraphics[width=.8\textwidth,height=.3\textheight]{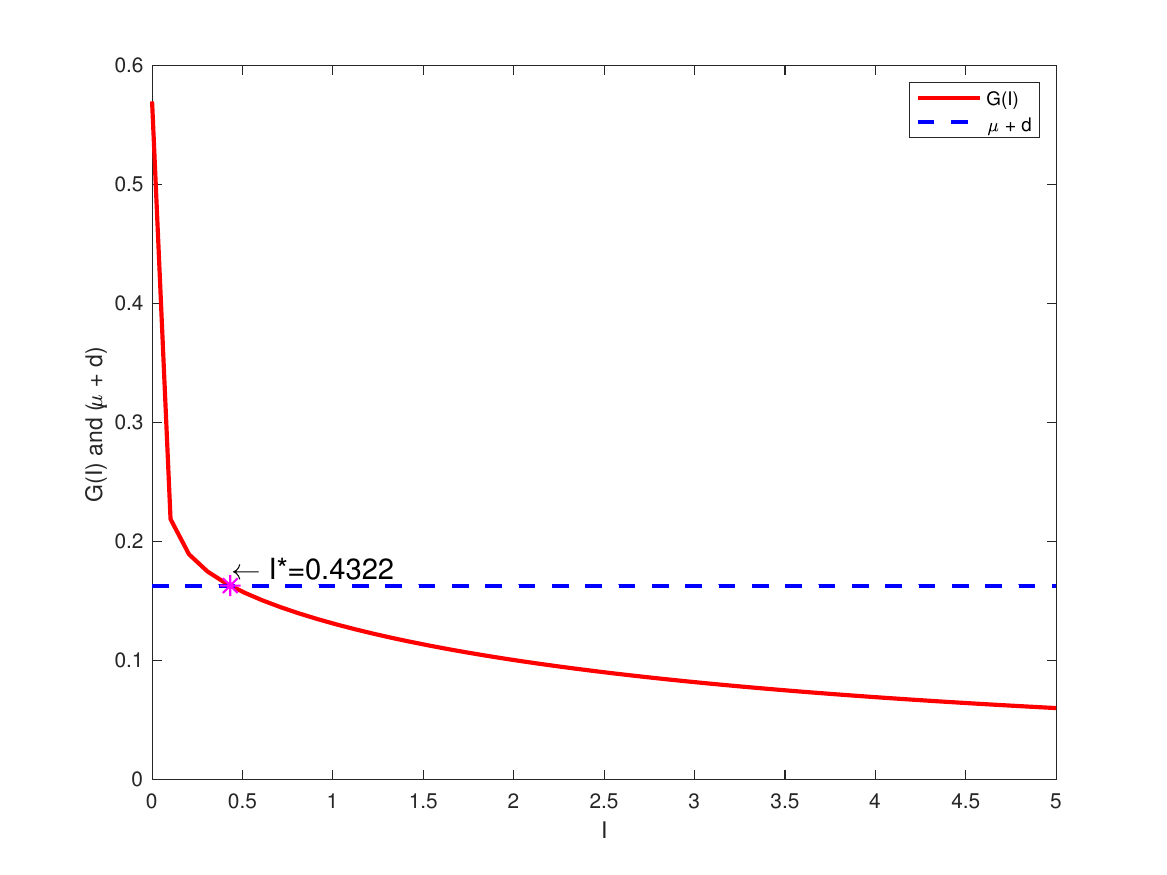}
    \caption{Graphs of $G(I)$ and $\mu + d$ together with their intersection point, $I^*$.}
    \label{fig:Istar}
\end{figure}

We note that while the expression of the disease-free equilibrium depends only on the total rate of entry to the $S_0$ group  and the natural death rate $\mu$, the endemic equilibrium however is a function of all the parameters of the model. 
 In reality, health officials are concerned with finding ways to control the outbreak of a disease so as to minimize its impact on the population, and if possible completely eradicate it. In most cases, the concerned population is  provided with more information on the causes and ways to limit infections. Whence it becomes relevant to know how the parameters of the model affect the size of the infected class. In particular,  studying  the effect of the information or education dissemination on the control of the disease is of interest. Theorem \ref{tm-stability-of-DFE} shows that the (local) stability of the disease-free equilibrium is independent of the amount of education  available about the disease, and the different infections ($S_1$ and $S_2$ classes) and education rates. In order to provide some information on the effect of education on the control of the disease, we decide to study how the magnitude of the component $I^*$ of the endemic equilibrium solution, when exists, is affected by the change in populations behavior due to education dissemination. To this end we introduce the quantity
$$ 
\tau=\frac{q}{\eta}
$$
which measures the ``{\it effective rate of education dissemination}" in the sense that $\tau>1$ means that there is sufficient information about the disease relative to the available resources and,  $\tau<1$ means that there is less education relative to the available resources. As a result, it becomes natural to study the dependence of $I^*$ with respect to $\tau$ when  $\beta_{0,u}>\mathcal{D}_0$. 
The following result holds.

\begin{tm}[Asymptotic behavior of $I^*$ with respect to $\tau=\frac{q}{\eta}$]\label{prop-2}
  Let $u\ge 0$ and suppose that $\beta_{0,u}=\max\{\beta_{j,u} : j=0,1,2 \}>\mathcal{D}_0$, and let ${\bf E^*}=(S_0^*,S_1^*,S_2^*,Z^*,I^*)^T$ denote the unique endemic equilibrium solution of \eqref{e-0-1}. The following conclusions hold. 
 \begin{itemize}
     \item [(i)]
  $I^*$ is non-increasing  with respect to $\tau=\frac{q}{\eta}$ and 
  \begin{equation}\label{asymp-of-I-star-1}
     \frac{B(\beta_{0,u}-\mathcal{D}_0)}{(\mu+d)(\beta_{0,u}+\frac{q\gamma}{\eta})}  \leq I^*  \le \min\left\{\frac{B(\beta_{0,u}-\mathcal{D}_0)}{\mu\beta_{0,u}},\frac{\mu}{\mathcal{D}_0}\right\}.
  \end{equation}
  Furthermore, if $\beta_{0,u}>\min\{\beta_{1,u},\beta_{2,u}\}$ or $\gamma_0>0$   then $I^*$ is strictly monotone decreasing with respect to $\tau=\frac{q}{\eta}$. 
  \item[(ii)] If  $  \mathcal{D}_0\geq \frac{1}{\gamma}\sum_{j=1}^2\beta_{j,u}\gamma_j$ 
  then 
 \begin{equation}\label{asymp-of-I-star}
     \lim_{\frac{q}{\eta}\to\infty}I^*=0.
 \end{equation}
 In addition, if $\mathcal{D}_0> \frac{1}{\gamma}\sum_{j=1}^2\beta_{j,u}\gamma_j$  
 then there is a positive constant $C>0$ such that 
 \begin{equation}\label{asymp-of-I-star-3}
   \frac{B(\beta_{0,u}-\mathcal{D}_0)}{(\mu+d)(\beta_{0,u}+\frac{q\gamma}{\eta})} \leq    I^*\leq \frac{C}{q/\eta} \quad \quad \forall\ \frac{q}{\eta}>0.
 \end{equation}
 \item[(iii)] If  $ \mathcal{D}_0<\frac{1}{\gamma}\sum_{j=1}^2\beta_{j,u}\gamma_j$, 
 then 
 \begin{equation}
     \lim_{\frac{q}{\eta}\to\infty}I^*=I^*_{\infty}>0
 \end{equation}
 where $I^*_\infty$ is the unique positive solution of the equation $ 
 \mathcal{D}_0=\frac{\mu}{\gamma}\sum_{j=1}^2\frac{\beta_{j,u}\gamma_j}{\mu+\beta_{j,u}I^*_\infty}$. 
 \end{itemize}
\end{tm}

Theorem \ref{prop-2} confirms that an appropriate  education level of the population about the disease will always help to contain and minimize significantly its impact. In fact Theorem \ref{prop-2} (ii) shows $I^*$ converges to $0$ at the same rate as $(\frac{q}{\eta})^{-1}$ does as $\frac{q}{\eta}\to\infty$ when hypothesis $\mathcal{D}_0>\frac{1}{\gamma}\sum_{j=1}^2\beta_{j,u}\gamma_j$ holds. One biological interpretation of this hypothesis is that the {\it relative infection weight} $\frac{1}{\gamma}\sum_{j=1}^2\beta_{j,u}\gamma_j$ of the susceptible classes $S_1$ and $S_2$ is smaller than the critical death rate $\mathcal{D}_0$. In such scenario, the size of the infected portion of the population at the endemic equilibrium  will be controlled significantly. On the other hand, according to Theorem \ref{prop-2} (iii), we see that if the relative infection weight $\frac{1}{\gamma}\sum_{j=1}^2\beta_{j,u}\gamma_j$ of the susceptible classes $S_1$ and $S_2$ is greater than  the critical death rate $\mathcal{D}_0$, then  the size of the infected portion of the population at the endemic equilibrium will remain uniformly bounded away from zero as a function of $\tau=\frac{q}{\eta}$. 

As a consequence of Theorem \ref{prop-2}, the following result on the asymptotic behavior of the susceptible population of ${\bf E}^*$ with respect to $\tau=\frac{q}{\eta}$ can be proved.

\begin{tm}
[Asymptotic behavior of $(S_0^*,S_1^*,S_2^*)$ with respect to $\tau=\frac{q}{\eta}$]\label{prop-3}
 Suppose that ${\beta_{0,u}=\max\{\beta_{0,u},\beta_{1,u},\beta_{2,u}\}}>\mathcal{D}_0$, and let ${\bf E^*}=(S_0^*,S_1^*,S_2^*,Z^*,I^*)^T$ denote the unique endemic equilibrium solution of \eqref{e-0-1}. The following holds. 
 \begin{itemize}
  \item[(i)] If  $ \mathcal{D}_0= \frac{1}{\gamma}\sum_{j=1}^2\beta_{j,u}\gamma_j$, 
 then 
 \begin{equation}\label{asymp-of-S-star}
     \lim_{\frac{q}{\eta}\to\infty}(S_0^*,S_1^*,S_2^*)=\Big(0,\frac{\gamma_1B}{\mu\gamma},\frac{\gamma_2B}{\mu\gamma}\Big).
 \end{equation}
 \item[(ii)] If  $\mathcal{D}_0> \frac{1}{\gamma}\sum_{j=1}^2\beta_{j,u}\gamma_j$, 
 then
 \begin{equation}\label{asymp-of-S-star-3}
   \lim_{\frac{q}{\eta}\to\infty}(S_0^*,S_1^*,S_2^*)=\left(\frac{B}{\mu+\gamma Z^*_\infty},\frac{\gamma_1BZ^*_\infty}{\mu(\mu+\gamma Z^*_{\infty})},\frac{\gamma_2BZ^*_\infty}{\mu(\mu+\gamma Z^*_\infty)}\right),  
 \end{equation}
 where $Z^*_\infty$ is the unique solution of the algebraic equation $ 
 \mathcal{D}_0=\frac{\mu\big(\beta_{0,u}+\frac{Z^*_\infty}{\mu}\sum_{j=1}^2\beta_{j,u}\gamma_j\big)}{\mu+\gamma Z^*_\infty}.$
 \item[(iii)] If $ \mathcal{D}_0<\frac{1}{\gamma}\sum_{j=1}^2\beta_{j,u}\gamma_j$ and, by setting $I^*_\infty=\lim_{\frac{q}{\eta}\to\infty}I^*>0$, it holds that 
 \begin{equation}\label{Asymp-S-star}
      \lim_{\frac{q}{\eta}\to\infty}(S_0^*,S_1^*,S_2^*)=\left(0,\frac{\gamma_1B}{\gamma(\mu+\beta_{1,u} I^*_{\infty})},\frac{\gamma_2B}{\gamma(\mu+\beta_{2,u}I^*_\infty)}\right). 
 \end{equation}
  
 \end{itemize}\end{tm}



 We have seen from Theorem \ref{tm-stability-of-DFE} that when $\max\{\beta_{0,u},\beta_{1,u},\beta_{2,u}\}\le \mathcal{D}_0$ then the disease will be eventually eradicated from the population. We also note from Theorem \ref{tm-existence-of-EE} that \eqref{e-0-1} has a unique endemic equilibrium ${\bf E^*}$ whenever $\beta_{0,u}>\mathcal{D}_0$. The following result is about the linear stability of the endemic equilibrium for $\beta_{0,0}>\mathcal{D}_0$.
 
 
 \medskip
 
 \begin{tm}\label{Th-5}
 Let $u=0$ and suppose that
 \begin{equation}\label{z-z-z-3} \mathcal{D}_0\left(\mathcal{D}_0+\frac{q}{\eta}\gamma\right)\ge \frac{q}{\eta}\sum_{j=1}^2\beta_{j,0}\gamma_j. 
 \end{equation}
 There is $\beta_{0,0,\max}>\mathcal{D}_0$ such that if $\beta_{0,0}\in(\mathcal{D}_0,\beta_{0,0,\max})$, then the endemic equilibrium solution ${\bf E}^*$ is linearly stable. Moreover if $\beta_{0,0,\max}<\infty$, then ${\bf E}^*$ is not linearly stable at $\beta_{0,0}=\beta_{0,0,\max}$. 
 \end{tm}
 It is not clear whether $\beta_{0,0,\max}$ can be taken to be infinity. We plan to further explore this question in our future work.  Next, we discuss whether the disease will become permanent in the population whenever $\beta_{0,u}>\mathcal{D}_0$.

 \begin{tm}[Permanence of disease] \label{tm-uniform-persistency}
Let $u\ge 0$ and suppose that $\beta_{0,u}=\max\{\beta_{0,u},\beta_{1,u},\beta_{2,u}\}>\mathcal{D}_0$. There exists a positive constant $\overline{m}_u>0$ such that any solution  $(S_0(t),S_1(t),S_2(t),Z(t),I(t))$ of \eqref{e-0-1} with initial in $\mathcal{X}_u$  satisfies   
\begin{equation}\label{persistent-eq-1}
    \frac{\mu(\beta_{0,u}-\mathcal{D}_0)}{\mathcal{D}_0(\beta_{0,u}+\frac{q}{\eta}\gamma)}\leq \limsup_{t\to\infty}I(t)\leq \frac{(\mu+d)(\beta_{0,u}-\mathcal{D}_0)}{\mathcal{D}_0\beta_{0,u}}
\end{equation}
and 
\begin{equation}\label{persistent-eq}
    \liminf_{t\to\infty}I(t)>\overline{m}_u.
\end{equation}
\end{tm}
 We note that the eventual a priori upper bound $\frac{(\mu+d)(\beta_{0,u}-\mathcal{D}_0)}{\mathcal{D}_0\beta_{0,u}}$ in \eqref{persistent-eq-1} for positive solutions decreases with respect to the time delay $u$. This is because the function $\beta_{0,u}$ is decreasing in $u$ and the function $(\mathcal{D}_0,\infty)\ni\beta\mapsto \frac{(\mu+d)(\beta-\mathcal{D}_0)}{\mathcal{D}_0\beta}$ is increasing. Hence we conclude that increasing the time delay has a positive effect in controlling the size of the infected population as time evolves. In fact, our numerical simulations in Section~\ref{sec3} confirm that the infection rates decrease as the time delay $u$ increases.    

The rest of the paper is organized as follows. In section \ref{sec0}, we study the global existence, uniqueness and positivity of the classical solutions. Section \ref{sec1} is devoted to the proof of Theorem \ref{tm-stability-of-DFE}. The proofs of Theorems \ref{tm-existence-of-EE}, \ref{prop-2}, and ~\ref{Th-5} are provided in Section~\ref{sec2}  while the proof of Theorem \ref{tm-uniform-persistency} is presented in section~\ref{Sec-for-persistence}. The numerical investigation of the epidemic model is completed in Section \ref{sec3}. 

\section{Basic properties of solutions}\label{sec0}
In the current section, we discuss the global existence, uniqueness, and positivity of classical solutions. To this end, consider the mappings defined on $\mathbb{R}^3\times\mathbb{R}\times\mathbb{R}$  by
$$ 
\mathcal{H}(S_0,S_1,S_2,Z,I)=
\left( 
\aligned
B-\gamma S_0Z-\mu S_0\\
\gamma_1S_0Z-\mu S_1\quad \\
\gamma_2S_0Z-\mu S_2\quad \\
qI-\eta Z\qquad \\
-\mu I-dI\qquad
\endaligned
\right)
\quad 
\text{and} 
\quad 
\mathcal{Q}(S_0,S_1,S_2,Z,I)=
\left( 
\aligned
-\beta_{0,u}S_0I\quad\\
-\beta_{1,u}S_1I\quad \\
-\beta_{2,u}S_2I\quad\\
0\qquad\\
\sum_{i=0}^2\beta_{i,u}S_iI
\endaligned
\right).
$$
Hence the delay-differential equation \eqref{e-0-1} can be written as 
\begin{small}
\begin{equation}\label{plq1}
    \frac{d}{dt}
    (S_0,S_1,S_2,Z,I)^T
    = \mathcal{H}(S_0(t),S_1(t),S_2(t),Z(t),I(t))+\mathcal{Q}(S_0(t),S_1(t),S_2(t),Z(t),I(t-u))\ \ t>0,
\end{equation}
\end{small}
with initial data satisfying \eqref{initial-cond}. Since the functions $\mathcal{H}$ and $\mathcal{Q}$ are of class $C^1$ and $I(0+\cdot)\in C([-u,0]:\mathbb{R}_+)$,  then it follows immediately from standard theory of delay ODEs (e.g., \cite{Hale77}) that the initial value problem \eqref{plq1} has a unique local solution $(S_0(t),S_1(t),S_2(t),Z(t),I(t+\cdot))$ defined on a maximal interval of existence $(0,T_{\max}]$
 for some $T_{\max}\in(0,\infty]$. Moreover, if $T_{\max}<\infty$ then 
 \begin{equation}\label{plq2}
     \limsup_{t\to T_{\max}}\Big[\sum_{i=0}^3|S_i(t)|+|Z(t)|+\|I(t+\cdot)\|_{C([-u,0]\ :\ \mathbb{R}_+)}\Big]=\infty.
 \end{equation}
{\bf Positivity  and uniform boundedness of Solution.} We first suppose that 
\begin{equation}\label{revis-3}
    \min\{S_0(0),S_1(0),S_2(0),Z(0),\min_{s\in[-u,0]}I(s)\}>0.
\end{equation}
Then by continuity of the map
\begin{equation}\label{pl212}
 [0,T_{\max})\ni t\mapsto \min\{S_0(t),S_1(t),S_2(t),Z(t),\min_{s\in[-u,0]}I(t+s)\},
\end{equation}
 there is $ 0< T'<T_{\max}$ such that
\begin{equation}\label{revis-4}
    \min\{S_0(t),S_1(t),S_2(t),Z(t),\min_{s\in[-u,0]}I(t+s)\}>0 \quad \forall\ t\in[0,T'].
\end{equation}
Let $T'_{\max}$ denote the supremum of positive numbers in $(0,T_{\max})$ satisfying \eqref{revis-4}. We then show that $T'_{\max}=T_{\max}$. Suppose on the contrary that $T'_{\max}<T_{\max}$. Let 
$$ 
M=\max_{t\in[0,T'_{\max}]}\{Z(t),\max_{s\in[-u,0]}I(t+s)\}>0.
$$
Hence, by definition of $T'_{\max}$, we have 
$$ 
\frac{d}{dt}S_0\geq B-(M+1)\left(\sum_{i=0}^2\gamma_i+\beta_{0,u}+\mu
\right)S_0\quad 0<t\le T'_{\max},
$$
and hence by comparison principle for ODEs it holds that
$$ 
S_0(t)\ge l_0:= \min\left\{S_0(0),\frac{B}{(M+1)\big(\sum_{i=0}^2\gamma_i+\beta_{0,u}+\mu
\big)}\right\} \quad \forall\ t\in[0,T'_{\max}].
$$
For each $j=1,2$, we have 
$$ 
\frac{d}{dt}S_j(t)\geq -(M\beta_{j,u}+\mu)S_j\quad 0<t<T'_{\max}
$$
and hence by comparison principle for ODEs it holds that
$$ 
S_j(t)\ge l_j:= e^{-(M\beta_{j,u}+\mu)T'_{\max}}S_j(0) \quad \forall\ t\in[0,T'_{\max}].
$$
From the equation of $Z(t)$, we have 
$$ 
\frac{d}{dt}Z(t)\ge -\eta Z(t) \quad 0<t<T'_{\max},
$$
which implies that 
$$ 
Z(t)\geq l_{Z}:=e^{-\eta T'_{\max}}Z(0)\quad \forall\ t\in[0,T'_{\max}].
$$
Observe also that 
$$ 
\frac{d}{dt}I\ge (\mu+d)I \quad \forall\ t\in(0,T'_{\max}],
$$
which also implies that
$$ 
I(t)\ge e^{-(\mu+d)T'_{\max}}I(0) \quad \forall\ t\in[0,T'_{\max}].
$$
Observe from the last inequality and the definition of $T'_{\max}$ that $I(t-s)>0$ for every $s\in[-u,0]$ and $t \in[0,T'_{\max}]$. Thus, by continuity of the map $[0,T'_{\max}]\times[-u,0]\ni (t,s)\mapsto I(t-s)$ and compactness of the set $[0,T'_{\max}]\times[-u,0]$  we have that 
$$ 
l_I:=\min_{(t,s)[0,T'_{\max}]\times[-u,0]}I(t-s)>0.
$$
Therefore, we obtain that
$$ 
 \min\{S_0(T'_{\max}),S_1(T'_{\max}),S_2(T'_{\max}),Z(T'_{\max}),\min_{s\in[-u,0]}I(T'_{\max}-s)\}\ge \min\{l_0,l_1,l_2,l_Z,l_I\}>0.
$$
This in turn together with the continuity of the mapping defined in \eqref{pl212} imply that there is some $T''\in(T'_{\max},T_{\max})$ such that \eqref{revis-4} holds for every $t\in[0,T'']$, which contradicts the definition of $T'_{\max}$. Therefore, we must have that $T'_{\max}=T_{\max}$.

\quad Next, since the solution is positive, it follows from the arguments used to establish \eqref{x-0} and \eqref{x0-1} that 
$$ 
N(t)=\sum_{i=0}^2S_i(t)+I(t)\leq \max\left\{N(0),\frac{B}{\mu}\right\} \quad \forall\ t\in[0,T_{\max})
$$
and 
$$ 
Z(t)\leq \max\left\{Z(0), \frac{q}{\eta}\max\left\{N(0),\frac{B}{\mu}\right\}\right\}\quad \forall\ t\in[0,T_{\max}).
$$
Therefore, in view of \eqref{plq2}, we conclude that $T_{\max}=\infty$.

We note from the above discussion that the uniform boundedness of solutions is guaranteed once the solution remains nonnegative. Now, let $(S_0(0),S_1(0),S_2(0),Z(0),I(0+\cdot))$ be a nonnegative initial data  and $T_{\max}\in(0,\infty]$ be the maximal time of existence of the associated unique classical solution $(S_0(t),S_1(t),S_2(t),Z(t),I(t))$.  Consider the sequence of positive initial data
$$ 
\left(S_0(0)+\frac{1}{n},S_1(0)+\frac{1}{n},S_2(0)+\frac{1}{n},Z(0)+\frac{1}{n},I(0+\cdot)+\frac{1}{n}\right)
$$
and 
$(S_0^n(t),S_1^n(t),S_2^n(t),Z^n(t),I^n(t))$ denote the associated unique classical solution. Since $ 
(S_0(0)+\frac{1}{n},S_1(0)+\frac{1}{n},S_2(0)+\frac{1}{n},Z(0)+\frac{1}{n},I(0+\cdot)+\frac{1}{n})
$ satisfy \eqref{revis-3} for each $n$, then \newline $(S_0^n(t),S_1^n(t),S_2^n(t),Z^n(t),I^n(t))$ is defined for all time $t>0$ with 
$$ 
\min\{S_0^n(t),S_1^n(t),S_2^n(t),Z^n(t),I^n(t)\}>0 \quad \forall\ t>0.
$$
 But since 
 $$ 
 \lim_{n\to\infty}\left(S_0(0)+\frac{1}{n},S_1(0)+\frac{1}{n},S_2(0)+\frac{1}{n},Z(0)+\frac{1}{n},I(0)+\frac{1}{n}\right)=(S_0(0),S_1(0),S_2(0),Z(0),I(0))
 $$
 then 
$$ 
\lim_{n\to\infty}(S_0^n(t),S_1^n(t),S_2^n(t),Z^n(t),I^n(t))=(S_0(t),S_1(t),S_2(t),Z(t),I(t))\quad \forall\ t\in(0,T_{\max}),
$$
which imply that 
$$ 
\min\{S_0(t),S_1(t),S_2(t),Z(t),I(t)\}\ge 0 \quad \forall\ t\in(0,T_{\max}).
$$
Therefore, as noted in the above, this last inequality also implies that $T_{\max}=\infty$.

\section{Analysis of the disease-free equilibrium }\label{sec1}
In this section we study the stability of the disease-free equilibrium. Equilibrium solutions are solutions of the system of algebraic equations obtained by setting the right hand sides of  \eqref{e-0-1} equal to zeros.  At times throughout the rest of the manuscript, to simplify the notations, we will suppress the dependence of the infection rates on $u$ unless otherwise stated and write $\beta_{j}=\beta_{j,u}$ for each $j=0,1,2$. Note that this will not cause any confusion in our presentation.  
 The linearized system of \eqref{e-0-1} at an equilibrium point $\tilde{\bf E} =(\tilde{S}_0,\tilde{S}_1,\tilde{S}_2,\tilde{Z},\tilde{I})^T$  is \begin{equation}\label{linear-syst}
     \begin{cases}
         \frac{d}{dt}S_0=-(\mu+\beta_0\tilde{I}+\gamma\tilde{Z})S_0(t)-\beta_0\tilde{S}_0I(t-u)-\gamma\tilde{S}_0 Z(t)\cr
         \frac{d}{dt}S_j=\gamma_j\tilde{Z}S_0-(\mu+\beta_j\tilde{I})S_j(t)-\beta_j\tilde{S}_jI(t-u)+\gamma_j\tilde{S}_0(t)Z(t),\ j=1,2\cr
         \frac{d}{dt}Z=qI(t)-\eta Z(t) \cr
         \frac{d}{dt} I= \beta_0\tilde{I}S_0(t)+\beta_1\tilde{I}S_1(t)+\beta_2\tilde{I}S_2(t) +\sum_{j=0}^2\beta_j\tilde{{S}_j}I(t-u) -\frac{\mathcal{D}_0B}{\mu}I(t)
     \end{cases}
 \end{equation} 
where $\gamma=\sum_{j=0}^2\gamma_j$. The characteristic equation of \eqref{linear-syst} is
\begin{equation}\label{p-0}
    0= P_u(\lambda;{\bf\tilde{E}}):=\text{det}(\lambda \mathcal{I}-\mathcal{M}_u(\lambda;{\bf\tilde{E}}))
\end{equation}
where  $\mathcal{I}$ denotes the identity matrix and 
\begin{small}
$$ 
\mathcal{M}_u(\lambda;{\bf\tilde{E}}):=\left[ 
\begin{array}{ccccc}
  -B/\tilde{S}_0  & 0 & 0 & -\gamma\tilde{S}_0 & -\beta_0\tilde{S}_0e^{-\lambda u}  \\ 
    \gamma_1\tilde{Z} & -(\mu +\beta_1\tilde{I}) & 0  & \gamma_1\tilde{S}_0 & -\beta_1\tilde{S}_1 e^{-\lambda u} \\
  \gamma_2\tilde{Z} & 0 & -(\mu+\beta_2\tilde{I})  & \gamma_2\tilde{S}_0 & -\beta_2\tilde{S}_2e^{-\lambda u} \\
  0 & 0 & 0 & -\eta& q\\
  \beta_0\tilde{I} & \beta_1\tilde{I} & \beta_2\tilde{I}& 0  & \sum_{j=0}^2\beta_j\tilde{S}_je^{-\lambda u}-\frac{\mathcal{D}_0B}{\mu} 
\end{array}
\right].
$$
\end{small} Observe that 
\begin{equation}\label{p-1}
P_u(\lambda;{\bf E^0})=\left(\lambda-\frac{\beta_0 B}{\mu}e^{-\lambda u}+\frac{\mathcal{D}_0B}{\mu}\right)(\lambda+\mu)^3(\lambda+\eta).
\end{equation}
We present the following lemma.
\begin{lem}\label{lem-1-0-01}
The disease-free equilibrium ${\bf E}^0$ is linearly stable if and only if $\beta_0<\mathcal{D}_0$.
\end{lem}
\begin{proof} It is clear from \eqref{p-1} that $\lambda=-\mu$ and $\lambda=-\eta$ are always roots of the characteristic equation $P_u(\lambda;{\bf E^0})=0$. Thus the linear stability of ${\bf E^0}$ is determined by the signs of the real parts of the roots of the equation
\begin{equation}\label{p-3}
    0=\lambda+\frac{\mathcal{D}_0B}{\mu}-\frac{\beta_0 B}{\mu}e^{-\lambda u}:=Q(\lambda)-\frac{\beta_0 B}{\mu}e^{-\lambda u}.
\end{equation}
Observe that $Q(-\frac{\mathcal{D}_0B}{\mu})=0$ and $Q(0)=\frac{\mathcal{D}_0B}{\mu}$. \\ 
{\bf Case 1.}  $\mathcal{D}_0>\beta_0$. Since $\lambda=-\frac{\mathcal{D}_0B}{\mu}<0$ is the only root of $Q(\lambda)=0$ and $ Q(0)>\frac{\beta_0 B}{\mu}$, then all roots of \eqref{p-3} have negative real part. Whence ${\bf E^0}$ is linearly stable if $\beta_0<\mathcal{D}_0$. \\
{\bf Case 2.}  $\beta_0=\mathcal{D}_0$. It is clear that $\lambda=0$ is a solution of \eqref{p-3}. Hence ${\bf E^0}$ is not linearly stable. \\ 
{\bf Case 3.}  $\mathcal{D}_0<\beta_0$. Now, observe that  $$\lim_{\lambda\to\infty}Q(\lambda)=+\infty \quad\text{and} \quad  \limsup_{\lambda\to\infty}\frac{\beta_0B}{\mu}e^{-\lambda u}\le \frac{\beta_0B}{\mu}.$$ 
Hence the intermediate value theorem  guarantees that there exists a positive real number $\lambda_u^0$ satisfying the equation \eqref{p-3}. Thus ${\bf E^0}$ is not linearly stable. 
\end{proof}

The first statement of Theorem \ref{tm-stability-of-DFE}  follows from the above result. It remains to show the non-linear stability of ${\bf E^0}$ when $\beta_0\le \mathcal{D}_0$.   We first prove the following general result about solution of the Cauchy problem.

\begin{lem}\label{lem-1-0-2}
Let $(S_0(t),S_1(t),S_2(t),Z(t),I(t))$ be a positive solution of \eqref{e-0-1} with initial in $\mathcal{X}_u$ and  let $\beta=\max\{\beta_0,\beta_1,\beta_2\}$. Then
\begin{equation}\label{amz-1}
    \overline{I}\left(\beta\overline{I}-\frac{B}{\mu}(\beta-\mathcal{D}_0)\right)\leq 0,
\end{equation}
where $\overline{I}=\mathlarger{\limsup_{t\to\infty}}I(t)$.
\end{lem}
\begin{proof}
Recall that
$$ 
\sum_{j=0}^{2}S_j(t)+I(t)\leq \frac{B}{\mu}\quad \forall\ t\ge 0.
$$
Hence
$$ 
\sum_{j=0}^2\beta_jS_j\leq \beta\sum_{j=0}^2S_j\leq \beta\left(\frac{B}{\mu}-I(t)\right),  \quad \forall\ t\geq 0.
$$
As a result, we obtain that 
$$ 
\frac{d}{dt}I(t)\leq \beta\left(\frac{B}{\mu}-I(t)\right)I(t-u)-(\mu+d)I(t), \quad \forall\ t> 0.
$$
Now, using the definition of limsup, for every $0<\lambda\ll 1$, there exists $t_{\lambda}\gg 0$ such that 
$$ 
I(t-u)\leq \overline{I}+\lambda\quad \forall\ t\ge t_{\lambda}.
$$
From the last two inequalities it follows that 
$$ 
\frac{d}{dt}I\leq \frac{\beta B(\overline{I}+\lambda)}{\mu}-(\mu+d+\beta(\overline{I}+\lambda))I(t)\quad \forall\ t\ge t_{\lambda},
$$
since $\frac{B}{\mu}\geq I(t)$ for every $t\ge 0$.
Thus, by the comparison principle for ODEs, we obtain that 
$$ 
\overline{I}\leq \frac{\beta B(\overline{I}+\lambda)}{\mu (\mu+d+\beta (\overline{I}+\lambda))}.
$$
Letting $\lambda\to0$ yields
$$ 
\overline{I}\leq \frac{\beta B\overline{I}}{\mu (\mu+d+\beta\overline{I})}.
$$
Since, $\frac{\mathcal{D}_0B}{\mu}=(\mu+d)$ and $\mu+d+\beta_0\overline{I}>0$, the last inequality is equivalent to  $\overline{I}\left(\beta\overline{I}-\frac{B}{\mu}(\beta-\mathcal{D}_0)\right)\leq 0,$ as needed.
\end{proof}

Next we present the proof of Theorem \ref{tm-stability-of-DFE} by using the previous two lemmas. 

\begin{proof}[Proof of Theorem \ref{tm-stability-of-DFE}]
Thanks to Lemma \ref{lem-1-0-01}, it remains to show the nonlinear stability of ${\bf E}^0$ when $\beta =\max\{\beta_0,\beta_1,\beta_2\}\leq \mathcal{D}_0$. Let $(S_0(t),S_1(t),S_2(t),Z(t),I(t))$ be a positive solution of \eqref{e-0-1} with initials in $\mathcal{X}_u$. If   $\beta < \mathcal{D}_0$, inequality \eqref{amz-1} implies that $\overline{I}=0$. If $\beta =\mathcal{D}_0$, it follows again from \eqref{amz-1} that $\overline{I}^2=0$, which gives $\overline{I}=0$. Therefore, since $I(t)\ge 0$ for every $t\ge 0$, we conclude that $\mathlarger{\lim_{t\to\infty}I(t)}=0$.  Next, observe that $  \mathlarger{\limsup_{t\to\infty}Z(t)}\leq \frac{q}{\eta}\mathlarger{\limsup_{t\to\infty}I(t)}=0$. Hence $\mathlarger{\lim_{t\to\infty}}Z(t)=\mathlarger{\lim_{t\to\infty}}I(t)=0$. This in turn, implies that $$ \lim_{t\to\infty}(S_0(t),S_1(t),S_2(t))= \left(\frac{B}{\mu},0,0\right).$$ As a result, it follows that $\mathlarger{\lim_{t\to\infty}}(S_0(t),S_1(t),S_2(t),I(t),Z(t))={\bf E^0}$. This completes the proof of Theorem \ref{tm-stability-of-DFE}.
\end{proof}

\section{Analysis of the endemic equilibrium.}\label{sec2}
In this section we shall show that if  $\beta_{0,u}> \mathcal{D}_0$ then \eqref{e-0-1} has a unique endemic equilibrium point. Moreover in the next section, we show that the disease will be permanent irrespective of the amount of information and/or education disseminated about it whenever $\beta_{0,u}> \mathcal{D}_0$. Our first result is about the existence and uniqueness of the endemic equilibrium.  Recall that we have set  $\beta_{j}=\beta_{j,u}$ for each $j=0,1,2$ in order to simplify notations. We recall the auxiliary function 
\begin{equation}\label{g-defi}
G(I,\beta_0,\tau)=\frac{\beta_0B}{\mu +\big(\beta_0+\tau\gamma \big)I} +\frac{B\tau}{\mu+\big(\beta_0+\tau\gamma \big)I}\sum_{j=1}^2\frac{\beta_j\gamma_jI}{\mu +\beta_jI}\quad \text{for}\  I\ge 0,\ \tau>0,\ \beta_0>0,
\end{equation}
 where $\min\{\beta_0,\beta_1,\beta_2\}>0$. Note that 
  ${\bf\tilde{E}} =(\tilde{S}_0,\tilde{S}_1,\tilde{S}_2,\tilde{Z},\tilde{I})^T$ is an  equilibrium of \eqref{e-0-1} with $\tilde{I}\ne 0$  if and only if  $ G(\tilde{I},\beta_0,\frac{q}{\eta})=\mu+d=\frac{\mathcal{D}_0B}{\mu}$. The following  results hold.  
  
  \begin{lem}\label{lem-000-0}
  For every $\tau>0$ and $\beta_0\ge 0$ it holds that $\mathlarger{\sup_{I\ge 0} G(I,\beta_0,\tau)}$ is finite and achieved. Moreover, the function  $ 
[0,\infty)\ni \beta_0\mapsto \max_{I\ge 0}G(I,\beta_0,\tau) $ is strictly increasing and $ \mathlarger{\max_{I\ge 0} G(I,\beta_0,\tau)}>\frac{\mathcal{D}_0B}{\mu} $ for every $\beta_0>\mathcal{D}_0$. Therefore, the following quantity 
$$ 
\mathcal{D}_{0,\tau}:=\min\left\{\beta_0\ge 0 \ :\ \max_{I\ge 0}G(I,\beta_0,\tau)\ge \frac{\mathcal{D}_0B}{\mu} \right\}
$$
is well defined
  \end{lem}
  \begin{proof} It is clear that 
  $$ 
   \ \lim_{I\to\infty}G(I,\beta_0,\tau)=0 \ \text{and}\  G(I,\beta_0,\tau)\ge 0 \quad \forall\ \tau,I,\beta_0\ge 0.
  $$
  Thus, since $I\mapsto G(I,\beta_0,\tau)$ is continuous, it follows that $\mathlarger{\sup_{I\ge 0}G(I,\beta_0,\tau)}$ is finite and achieved. Observe that 
$$ 
\partial_{\beta_0}G=\frac{B\big(\mu+\tau\gamma I-\tau I\sum_{j=1}^2\frac{\beta_j\gamma_j I}{\mu+\beta_jI}\big)}{(\mu+(\beta_0+\gamma\tau)I)^2}\geq \frac{B\big(\mu+\tau\gamma I-\tau I\sum_{j=1}^2\gamma_j\big)}{(\mu+(\beta_0+\gamma\tau)I)^2}=\frac{B(\mu+\tau\gamma_0I)}{(\mu+(\beta_0+\gamma\tau)I)^2}>0
$$
for every $\beta_0\ge0$, $\tau>0$ and $I\ge 0$. Thus, we conclude that the function 
$$ 
[0,\infty)\ni \beta_0\mapsto \max_{I\ge 0}G(I,\beta_0,\tau)
$$
is strictly increasing. Note that $G(0,\beta_0,\tau)=\frac{\beta_0B}{\mu}\ge\frac{\mathcal{D}_0B}{\mu}$ for every $\beta_0\ge\mathcal{D}_0$. The result thus follows.
  \end{proof}
  
\begin{tm} \label{lem-1-0-0-0}
 For every   $\tau>0$  let $\mathcal{D}_{0,\tau}\leq \mathcal{D}_0$ be given by Lemma \ref{lem-000-0}. The following hold.
\begin{description}
\item[(i)] If $0<\beta_0< \mathcal{D}_{0,\tau}$ then the algebraic equation $G(I,\beta_0,\tau)=\frac{\mathcal{D}_0B}{\mu}$ has no nonnegative root.

\item[(ii)] If $\mathcal{D}_{0,\tau}<\beta_0<\mathcal{D}_0$ then the algebraic equation $G(I,\beta_0,\tau)=\frac{\mathcal{D}_0B}{\mu}$  has two positive roots $I_{-}(\beta_0,\tau)<I_{+}(\beta_0,\tau)$. Furthermore, $\partial_{I}G(I_{-}(\beta_0,\tau),\beta_0,\tau)>0$ and  $\partial_{I}G(I_{+}(\beta_0,\tau),\beta_0,\tau)<0$, and the functions $ (\beta_0,\tau)\mapsto I_{\pm}(\beta_0,\tau)$ are of class $C^1$.

\item[(iii)] If $\mathcal{D}_{0,\tau}<\mathcal{D}_0$ then for every $\beta_0\in\{\mathcal{D}_{0,\tau},\mathcal{D}_0\}$ there is a unique positive solution $I(\beta_0,\tau)$ of the algebraic equation $G(I,\beta_0,\tau)=\frac{\mathcal{D}_0B}{\mu}$. 

\item[(iv)] If $\beta_0>\mathcal{D}_0$ there is a unique positive root $I(\beta_0,\tau)$ of the algebraic equation  $G(I,\beta_0,\tau)=\frac{\mathcal{D}_0B}{\mu}$. The functions $\beta_0\mapsto I(\beta_0,\tau)$ and $\tau \mapsto I(\beta_0,\tau)$ are smooth and $\partial_{I}G(I(\beta_0,\tau),\beta_0,\tau)<0$.
\end{description} 
\end{tm}
\begin{proof}
$(i)$ It is clear from Lemma \ref{lem-000-0} that $\mathlarger{\max_{I\ge 0}G(I,\beta_0,\tau)}<\frac{\mathcal{D}_0B}{\mu}$ for every $0< \beta_0<\mathcal{D}_{0,\tau}$. Hence the result follows.

$(ii)$ Let $ \beta_0\in (\mathcal{D}_{0,\tau}, \mathcal{D}_0)$. Then $G(0,\beta_0,\tau) =\frac{\beta_0B}{\mu}<\frac{\mathcal{D}_0B}{\mu}<\mathlarger{\max_{I\ge 0}G(I,\beta_0,\tau)}$. This shows that $\mathlarger{\max_{I\ge 0}G(I,\beta_0,\tau)}$ is achieved at an interior point $ I_{\max}$. Next, observe that 
\begin{align}\label{azam-1}
 \partial_IG=&\frac{B}{\big( \mu + b_0I \big)^2 }\left( - \beta_0 (\beta_0 + \tau \gamma) + \tau\sum_{j=1}^2\frac{\gamma_j\beta_j(\mu^2-\beta_jb_0  I^2)}{\big(\mu+\beta_jI\big)^2} \right).
 \end{align}
  Hence 
$$
\partial_IG(I,\beta_0,\tau)=0 \Longleftrightarrow   \tau\sum_{j=1}^2\gamma_j\beta_j-\beta_0(\beta_0+\tau\gamma) = \tau\sum_{j=1}^2\frac{\gamma_j\beta_j(2\mu\beta_jI+(\beta_jI)^2+\beta_jb_0  I^2)}{\big(\mu+\beta_jI\big)^2}. 
$$
An easy computation shows that $[0,\infty)\ni I\mapsto \frac{\gamma_j\beta_j(2\mu\beta_jI+(\beta_jI)^2+\beta_jb_0  I^2)}{\big(\mu+\beta_jI\big)^2}$ is strictly increasing for every $j=1,2$. As a result, the equation $\partial_IG=0$ has exactly a unique positive root, which is $I_{\max}$. Moreover, the functions $(0,I_{\max})\ni I \mapsto G(I,\beta_0,\tau)$ and $(I_{\max},\infty)\ni I \mapsto G(I,\beta_0,\tau)$ are strictly increasing and decreasing, respectively. The intermediate value theorem then implies that  there exist unique elements $I_{-}(\beta_0,\tau)\in(0,I_{\max})$ and $I_{+}(\beta_0,\tau)\in(I_{\max},\infty)$ such that $G(I_{\pm}(\beta_0,\tau),\beta_0,\tau)=\frac{\mathcal{D}_0B}{\mu}$. Furthermore, $\partial_{I}G(I_{-}(\beta_0,\tau),\beta_0,\tau)>0$ and  $\partial_{I}G(I_{+}(\beta_0,\tau),\beta_0,\tau)<0$, and the implicit function theorem guarantees that both the functions $ (\beta_0,\tau)\mapsto I_{\pm}(\beta_0,\tau)$ are of class $C^1$.

$ (iii)$ If  $\mathcal{D}_0>\mathcal{D}_{0,\tau}$ then as in case $(ii)$ $\mathlarger{\max_{I\ge 0}G(I,\mathcal{D}_0,\tau)}$ is achieved at an interior point $I_{\max} > 0.$ In this case, we have $I_{-}(\mathcal{D}_0,\tau)= 0$ and $I_{+}(\mathcal{D}_0,\tau)\in (I_{\max},\infty)$ are the two roots of $G=\frac{\mathcal{D}_0B}{\mu}$. Similarly we have $I(\mathcal{D}_{0,\tau},\tau)=I_{\max}>0$.

$(iv)$ Suppose that $\beta_0>\mathcal{D}_0$. Note from above that $\frac{\mathcal{D}_0B}{\mu}<G(0,\beta_0,\tau)=\mathlarger{\min_{0\leq I\leq I_{\max}}G(I,\beta_0,\tau)}$. Thus, as in the above, we can employ the intermediate value theorem to infer that there is exactly a unique solution $I(\beta_0,\tau)>0$ of the equation  $G(I,\beta_0,\tau)=\frac{\mathcal{D}_0B}{\mu}$, with $I(\beta_0,\tau)\in(I_{\max},\infty).$ 
\end{proof}

 The next result complements Theorem \ref{lem-1-0-0-0} and provides necessary and sufficient condition on the parameters for which $\mathcal{D}_{0,\tau}<\mathcal{D}_0$.
 
 \begin{prop}\label{prop-p-1}
 For every $\tau>0$, let $\mathcal{D}_{0,\tau}$ be given by Lemma \ref{lem-000-0} and define   
 \begin{equation}\label{d-tau-equ}
 \mathcal{D}_{1,\tau}=\frac{\frac{2}{\gamma}\sum_{j=1}^2\gamma_j\beta_j}{1+\sqrt{1+\frac{4\tau}{(\tau\gamma)^2}\sum_{j=1}^2\beta_j\gamma_j}}.
 \end{equation} 
The function $\tau\mapsto\mathcal{D}_{1,\tau}$ is strictly increasing and  $\mathcal{D}_0=\mathcal{D}_{0,\tau}$ if and only if $\mathcal{D}_0\ge\mathcal{D}_{1,\tau}$. Therefore if
   $ \mathcal{D}_0\geq \frac{1}{\gamma}\sum_{j=1}^2\gamma_j\beta_j$  then $\mathcal{D}_0=\mathcal{D}_{0,\tau}$ for every $\tau>0$, while if  $ \mathcal{D}_0< \frac{1}{\gamma}\sum_{j=1}^2\gamma_j\beta_j$  there is $\tau^*>0$ such that $\mathcal{D}_0=\mathcal{D}_{0,\tau}$ for every $0<\tau<\tau^*$ and $\mathcal{D}_0>\mathcal{D}_{0,\tau}$ for every $\tau\ge\tau^*$.
 \end{prop}
 \begin{proof} It is easy to see that the function $\tau\mapsto\mathcal{D}_{1,\tau}$ is strictly increasing and that $\mathcal{D}_{1,\tau}$ is the only positive solution of the algebraic equation in $\beta_0$ : 
 $$ 
 \beta_0(\beta_0+\tau\gamma)=\tau\sum_{j=1}^2\gamma_j\beta_j.
 $$
Hence  $\mathcal{D}_0 \geq \mathcal{D}_{1,\tau}$ if and only if $\mathcal{D}_0(\mathcal{D}_0+\tau\gamma)\ge\mathcal{D}_{1,\tau}(\mathcal{D}_{1,\tau}+\tau\gamma)=\tau\sum_{j=1}^2\gamma_j\beta_j$ and it is seen from \eqref{azam-1} that 
 $$ 
 \partial_IG(I,\mathcal{D}_0,\tau)<\partial_IG(0,\mathcal{D}_0,\tau)\leq 0 \quad \forall\ I> 0,
 $$
 which implies that $\mathlarger{\max_{I\ge 0}G(I,\mathcal{D}_0,\tau)}=G(0,\mathcal{D}_0,\tau)=\frac{\mathcal{D}_0B}{\mu}$. Whence, we deduce from Lemma \ref{lem-000-0} that $\mathcal{D}_0=\mathcal{D}_{0,\tau}$ since the map $\beta_0\mapsto \mathlarger{\max_{I\ge 0}G(I,\mathcal{D}_0,\tau)} $ is strictly decreasing. 
 
 On the other hand, if  $\mathcal{D}_0< \mathcal{D}_{1,\tau}$, then using again \eqref{azam-1}, we obtain that 
 $ \partial_IG(0,\mathcal{D}_0,\tau)>0$. This implies that $ \mathlarger{\max_{I\ge 0}G(I,\mathcal{D}_0,\tau)}>G(0,\mathcal{D}_0,\tau)=\frac{\mathcal{D}_0B}{\mu}$. Hence by continuity, there is some $\tilde{\beta}_0<\mathcal{D}_0$ close enough such that $\mathlarger{\max_{I\ge 0}G(I,\tilde{\beta_0},\tau)}>\frac{\mathcal{D}_0B}{\mu}$. This shows that $\mathcal{D}_{0,\tau}\leq \tilde{\beta}_0<\mathcal{D}_0$. The proof of Proposition \ref{prop-p-1} is then complete. 
 \end{proof}

\begin{proof}[Proof of Theorem \ref{tm-existence-of-EE}] 
For  $\beta_{0,u}=\beta_0> \mathcal{D}_0$ and $\tau=\frac{q}{\eta}$, the result  follows from Theorem \ref{lem-1-0-0-0} $(iv)$.
\end{proof}

 \begin{rk}
 We note from the proof of proposition \ref{prop-p-1} that $\mathcal{D}_{0,\tau}= \mathcal{D}_0$ if and only if $\mathcal{D}_0(\mathcal{D}_0+\tau\gamma)\ge \tau\sum_{j=0}^2\beta_{j,u}\gamma_j$. Hence thanks to Theorem \ref{lem-1-0-0-0} the following conclusions hold. 
 \begin{enumerate}
     \item[(1)] Assume that $\mathcal{D}_0\big(\mathcal{D}_0+\frac{q}{\eta}\gamma\big)\geq \frac{q}{\eta}\sum_{j=0}^2\beta_{j,u}\gamma_j$. If $\beta_{0,u}\le \mathcal{D}_0$ then  \eqref{e-0-1} has no endemic equilibrium solution.  However, if $\beta_{0,u}>\mathcal{D}_0$ then \eqref{e-0-1} has a unique endemic equilibrium solution (see Figure \ref{FIG2}).
     \begin{figure}[h]
         \centering
         \includegraphics[width=0.7\textwidth,height=.2\textheight]{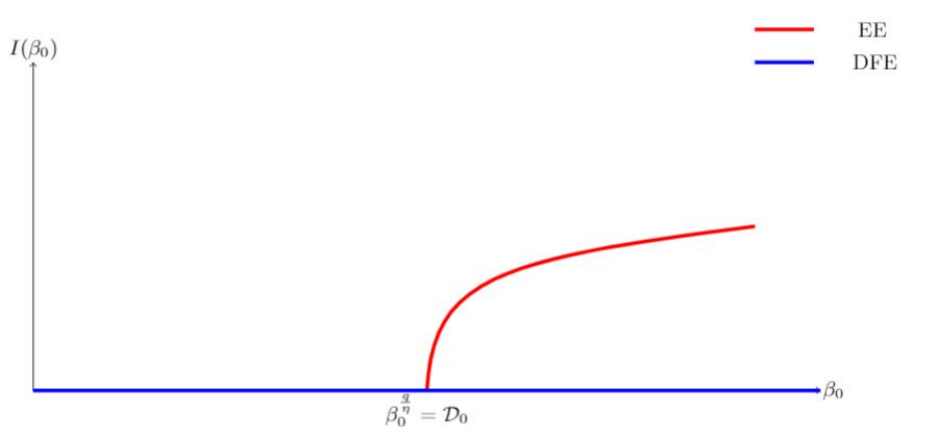}
         \caption{ Schematic graphs of $I$ component  of equilibrium solutions of \eqref{e-0-1} when $\mathcal{D}_0\big(\mathcal{D}_0+\frac{q}{\eta}\gamma\big)\geq \frac{q}{\eta}\sum_{j=0}^2\beta_{j,u}\gamma_j$. This corresponds to case $\mathcal{D}_{0,\frac{q}{\eta}}=\mathcal{D}_0$ in Theorem \ref{lem-1-0-0-0}.}
         \label{FIG2}
\end{figure}

\item[(2)] Assume that $\mathcal{D}_0\big(\mathcal{D}_0+\frac{q}{\eta}\gamma\big)< \frac{q}{\eta}\sum_{j=0}^2\beta_{j,u}\gamma_j$. Then there is $\mathcal{D}_{0,\frac{q}{\eta}}\in[0,\mathcal{D}_0)$ such that if $\mathcal{D}_{0,\frac{q}{\eta}}<\beta_{0,u}<\mathcal{D}_0$ then \eqref{e-0-1} has exactly two endemic equilibrium. And for every $\beta_{0,u}\ge \mathcal{D}_0$, \eqref{e-0-1} has one endemic equilibrium solution (see Figure \ref{FIG3}).
\begin{figure}[h]
         \centering
         \includegraphics[width=0.7\textwidth,height=.2\textheight]{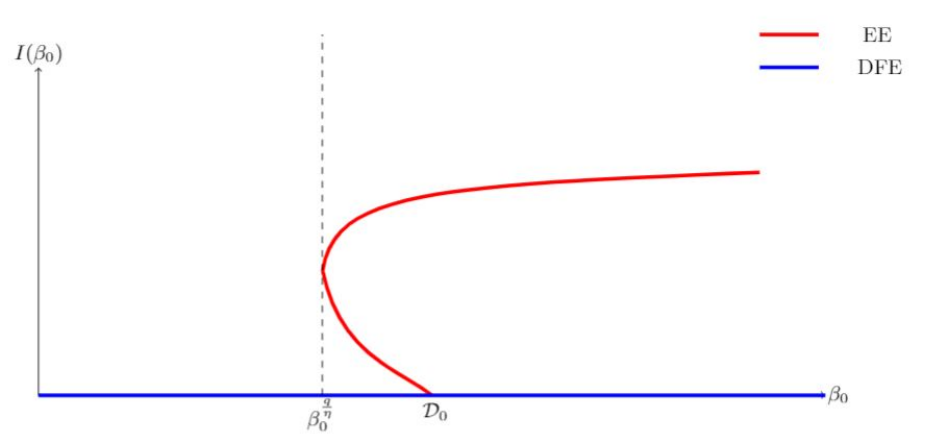}
         \caption{ Schematic graphs of equilibrium solutions of \eqref{e-0-1} when $\mathcal{D}_0\big(\mathcal{D}_0+\frac{q}{\eta}\gamma\big)< \frac{q}{\eta}\sum_{j=0}^2\beta_{j,u}\gamma_j$. This corresponds to the case $0<\mathcal{D}_{0,\frac{q}{\eta}}<\mathcal{D}_0$ in Theorem \ref{lem-1-0-0-0}.}
         \label{FIG3}
\end{figure}
\end{enumerate}
\end{rk}

\begin{proof}[Proof of Theorem \ref{prop-2}]
$(i)$ Recalling the function $G(I,\beta_0,\tau)$ 
 introduced in \eqref{g-defi}, we have 
 $$ 
 \frac{\partial }{\partial\tau }G= \frac{BI\left( \sum_{j=1}^2\frac{\mu\gamma_j(\beta_j-\beta_0)}{\mu+\beta_jI}-\gamma_0\beta_0 \right)}{(\mu +\beta_0I+\tau I\gamma)^2} \quad \forall\ \tau>0, I\ge 0.
 $$
 Whence, since $I^*$ is uniquely determined by the equation $G(I,\beta_0,\frac{q}{\eta})=\frac{\mathcal{D}_0B}{\mu}$ and $\beta_0\geq \max\{\beta_1,\beta_2\}$, the implicit function theorem implies that 
 $$ 
 \frac{\partial}{\partial \tau}I^*=-\frac{\partial_\tau G}{\partial_IG}\begin{cases}
 = 0 &  \text{if}\ \gamma_0= 0 \ \text{and}\ \beta_1=\beta_2=\beta_0\cr
 <0 & \text{if}\ \beta_0>\min\{\beta_1,\beta_2\} \ \text{or}\ \gamma_0> 0.
 \end{cases}\quad 
 $$ 
 Note that we have used Theorem \ref{lem-1-0-0-0} $(iv)$ to conclude that $\partial_IG(I^*,\beta_0,\frac{q}{\eta})<0$. This shows the monotonicity of $I^*$ with respect to $\frac{q}{\eta}$ as stated in the result.  A simple computation shows that 
 $$ 
 G\left(\frac{\mu(\beta_0-\mathcal{D}_0)}{\mathcal{D}_0(\beta_0+\frac{q\gamma}{\eta})},\beta_0,\tau\right)>\frac{\beta_0B}{\mu+\frac{\mu}{\mathcal{D}_0}(\mathcal{\beta}_0-\mathcal{D}_0)}=\frac{\mathcal{D}_0B}{\mu}\quad \forall\ \beta_0,\ \tau>0.
 $$
 
 By \eqref{azam-1}, we have
 \begin{align*}
 \partial_IG=&\frac{B}{\big( \mu + b_0I \big)^2 }\left( - \beta_0 b_0 + \tau\sum_{j=1}^2\frac{\gamma_j\beta_j(\mu^2-\beta_jb_0  I^2)}{\big(\mu+\beta_jI\big)^2} \right)\cr
 < & \frac{B}{\big( \mu + b_0I \big)^2 }\left( - \beta_0^2-\tau\sum_{j=0}^2\gamma_j\beta_0+ \tau\sum_{j=1}^2\gamma_j\beta_j \right)\le 0 \quad\forall\ I,\tau > 0 \quad \text{and} \ \beta_0\geq \max\{\beta_1,\beta_2\}.
 \end{align*}
 Hence, since $\beta_0=\max\{\beta_0,\beta_1,\beta_2\}$ then the function $ (0,\infty)\ni I\mapsto G(I,\beta_0,\frac{q}{\eta})$ is strictly decreasing, we conclude that the left hand side inequality of  \eqref{asymp-of-I-star-1} holds.  Observe that 
 $$ 
 B=\mu\sum_{j=0}^2S_{j}^*+\frac{\mathcal{D}_0B}{\mu}I^*+\gamma_0Z^*S_0^*.
 $$ 
 Hence $I^*\leq \frac{\mu}{\mathcal{D}_0}$. Thus the right hand side of \eqref{asymp-of-I-star-1} follows from \eqref{amz-1}.
 
 $(ii)$ Next we proceed by contradiction to show  that \eqref{asymp-of-I-star} holds under hypothesis $\mathcal{D}_0\geq \frac{1}{\gamma}\sum_{j=1}^2\beta_j\gamma_j$. We first note from the monotonicity  of $I^*(\beta_0,\frac{q}{\eta})$ with respect to $\frac{q}{\eta}$ that there is $I^*_\infty(\beta_0)\in[0,\frac{B}{\mu})$ such that 
 $$ \lim_{\tau\to\infty}I^*(\beta_0,\tau)=I^*_{\infty}(\beta_0).$$
 Now, we claim that $I^*_{\infty}:=I^*_{\infty}(\beta_0)=0$. If not, i.e, $I^*_\infty>0$, letting $\frac{q}{\eta}\to\infty$ in the equation $ 
 G(I^*,\beta_0,\frac{q}{\eta})= \frac{B \mathcal{D}_0}{\mu}$ gives 
 $$ 
 \frac{B\mathcal{D}_0}{\mu}=\frac{B\sum_{j=1}^2\frac{\beta_j\gamma_jI^*_\infty}{\mu+\beta_jI^*_\infty}}{I^*_\infty\gamma}=\frac{B}{\gamma}\sum_{j=1}^2\frac{\beta_j\gamma_j}{\mu+\beta_jI^*_\infty}<\frac{B}{\mu\gamma}\sum_{j=1}^2\beta_j\gamma_j
 $$
 which is in contradiction with hypothesis $\mathcal{D}_0\geq \frac{1}{\gamma}\sum_{j=1}^2\beta_j\gamma_j$. Thus \eqref{asymp-of-I-star} holds under hypothesis $\mathcal{D}_0\geq \frac{1}{\gamma}\sum_{j=1}^2\beta_j\gamma_j$. Next,  suppose that hypothesis $\mathcal{D}_0> \frac{1}{\gamma}\sum_{j=1}^2\beta_j\gamma_j$ holds and choose a positive constant $C_1\gg 1$ such that
 $$ 
 \frac{B}{\mu+\gamma C_1}\left( \beta_0+C_1\sum_{j=1}^2\frac{\beta_j\gamma_j}{\mu} \right)<\frac{\mathcal{D}_0 B}{\mu}.
 $$
 With this choice of the constant $C_1$ we obtain
 $$ 
 \lim_{\tau\to\infty}G(C_1\tau^{-1},\beta_0,\tau)= \frac{B}{\mu+\gamma C_1}\left( \beta_0+C_1\sum_{j=1}^2\frac{\beta_j\gamma_j}{\mu} \right)<\frac{\mathcal{D}_0B}{\mu}.
 $$
 Hence, there is $\tau_*\gg 1$ such that 
 $$ 
 I^*(\beta_0,\frac{q}{\eta})\leq \frac{C_1}{q/\eta} \quad \quad \forall\ \frac{q}{\eta}\geq \tau_*. 
 $$
 Therefore, since $I^*(\beta_0,\frac{q}{\eta})$ is uniformly bounded above by $\frac{B}{\mu}$,  we conclude that the right inequality of \eqref{asymp-of-I-star-3} holds with $C=\max\{C_1,\frac{B\tau_*}{\mu} \}$. Clearly the left inequality of \eqref{asymp-of-I-star-3} follows from \eqref{asymp-of-I-star-1}. 
 
 $(iii)$ Suppose that $ \mathcal{D}_0<\frac{1}{\gamma}\sum_{j=1}^2\beta_j\gamma_j$. Let $I^*_{\infty}(\beta_0)$ denote the unique positive solution of the algebraic equation
 $$ 
 \frac{\mathcal{D}_0B}{\mu}=\frac{B}{\gamma}\sum_{j=1}^2\frac{\beta_j\gamma_j}{\mu+\beta_jI^*_{\infty}}.
 $$
 For every $0<\varepsilon\ll 1$, using the fact that 
 $$ 
 \lim_{\tau\to\infty}G((1\pm\varepsilon)I^*_\infty,\beta_0,\tau)=\frac{B}{\gamma}\sum_{j=1}^2\frac{\beta_j\gamma_j}{\mu+(1\pm\varepsilon)\beta_jI^*_{\infty}},
 $$
 then it follows from the fact that $\frac{\mathcal{D}_0B}{\mu}=G(I^*,\beta_0,\frac{q}{\eta})$ and the fact that $I\mapsto G(I,\beta_0,\tau)$ is strictly decreasing that there is $\tau_\varepsilon\gg 0$ such that 
 $$ 
 (1-\varepsilon)I^*_{\infty}< I^*\left(\beta_0,\frac{q}{\eta}\right)<(1+\varepsilon)I^*_{\infty} \quad \forall\ \frac{q}{\eta}\ge \tau_{\varepsilon}.
 $$
 Thus we conclude that $\mathlarger{\lim_{\frac{q}{\eta}\to\infty}I^*\left(\beta_0,\frac{q}{\eta}\right)}=I^*_\infty(\beta_0)$.
 The proof of the Theorem is complete.
\end{proof}


The proof of Theorem \ref{prop-3} follows similar arguments as those used in the proof of Theorem \ref{prop-2}. So, in order to avoid repeating the same arguments, we do not include its proof. We end this section with the proof of Theorem \ref{Th-5}. Our method to prove this result is based on the bifurcation theory.  Suppose $u=0$ and $\beta_j=\beta_{j,0}$, $j=0,1,2$ and let ${\bf E}= (S_0,S_1,S_2,Z,I)^T $  and 
$F(\beta_0,{\bf E})$ be the vector field to the right hand side of \eqref{e-0-1}, and we consider the parameter $\beta_0$ as our bifurcation parameter. Note that the function $(\beta_0,{\bf E})\mapsto F(\beta_0,{\bf E})$ is of class $C^{\infty}$.  Recalling \eqref{linear-syst}, we have that 
$$ 
D_{\bf E}F(\beta_0,{\bf E}^0)=\mathcal{M}_{0}(0,{\bf E}^0), \quad \forall\ \beta_0>0
$$
and 
\begin{equation}\label{z-z-20}
\mathcal{K}:=D_{\beta_0{\bf E}}F(\beta_0,{\bf E}^0)=\left[
\begin{array}{ccccc}
    0 & 0 &  0  &0 & -\frac{B}{\mu}    \cr
    0 & 0 &  0   &     0           & 0  \cr 
    0 & 0 &  0   &     0           & 0   \cr
    0 & 0 &  0   &     0           & 0  \cr
    0 & 0 &  0   &     0           & \frac{B}{\mu}
\end{array}
\right]
\end{equation}
where $ D_{\bf E}F(\beta_0,{\bf E}^0)$ stands for the Jacobian matrix with respect to $ {\bf E}$ only and $ D_{\beta_0{\bf E}}F(\beta_0,{\bf E}^0)$ the partial derivative of the jacobian matrix with respect to $\beta_0$. Moreover, it follows from \eqref{p-1} that
$$ 
r(\beta_0):=\frac{\beta_0B}{\mu}-(\mu+d)=\frac{B}{\mu}(\beta_0-\mathcal{D}_0)\quad \quad \forall \beta_0>\mathcal{D}_0-\frac{\mu}{B}\min\{\mu,\eta\}
$$
is the maximal eigenvalue of $ D_{\bf E}F(\beta_0,{\bf E}^0)$. The following lemma shows that $r(\mathcal{D}_0)$ is a $\mathcal{K}-$simple eigenvalue of $D_{\bf E}F(\mathcal{D}_0,{\bf E}^0)$. We denote by $\{{\bf e}_0,{\bf e}_1,{\bf e}_2,{\bf e}_3,{\bf e}_4\}$  the canonical basis of $\mathbb{R}^5$. Given a 5 by 5 square matrix $M $, we denote by $\mathcal{N}(M)$ and $\mathcal{R}(M)$ the kernel and range of the linear operator induced by $M$ on $\mathbb{R}^5$, respectively. 

\begin{lem} \label{lem-3-2} Let $\beta_0> \mathcal{D}_0-\frac{\mu}{B}\min\{\mu,\eta\}$ and define $\eta(\beta_0)=\eta+r(\beta_0)$, $\mu(\beta_0)=\mu+r(\beta_0)$ and 
$$ 
{\bf J}^0(\beta_0):=\left(-\frac{B}{\mu\mu(\beta_0)}\left(\beta_0+\frac{\gamma q}{\eta(\beta_0)}\right),\frac{Bq\gamma_1}{\mu\mu(\beta_0)\eta(\beta_0))},\frac{Bq\gamma_2}{\mu\mu(\beta_0)\eta(\beta_0)},\frac{q}{\eta(\beta_0)},1\right)^T.$$
The following hold.
\begin{itemize}
\item[(i)] $r(\mathcal{D}_0)=0$ is a $\mathcal{K}-$simple eigenvalue of $D_{\bf E}F(\mathcal{D}_0,{\bf E}^0)$. In fact, it holds that  
$$ 
\mathcal{N}(D_{\bf E}F(\mathcal{D}_0,{\bf E}^0))=span\{ {\bf J}^0(\mathcal{D}_0)\} \quad \text{and}\quad \mathcal{K}{\bf J}^0(\mathcal{D}_0)\notin \mathcal{R}(D_{\bf E}F(\mathcal{D}_0,{\bf E}^0))=\mathbb{R}^4\times\{0\}. 
$$
\item[(ii)] $r(\beta_0)$ is a $\mathcal{I}-$simple eigenvalue of $D_{\bf E}F(\beta_0,{\bf E}^0)$. More precisely, it holds that 
$$ 
\mathcal{N}(D_{\bf E}F(\beta_0,{\bf E}^0)-r(\beta_0)\mathcal{I})=span\{ {\bf J}^0(\beta_0)\} \quad \text{and}\quad {\bf J}^0(\beta_0)\notin \mathcal{R}(D_{\bf E}F(\beta_0,{\bf E}^0))=\mathbb{R}^4\times\{0\}.
$$
\end{itemize}
\end{lem}
\begin{proof}$(i)$
It is easily seen that with $ 
{\bf J}^0(\mathcal{D}_0)=\Big(-\frac{B}{\mu^2}(\mathcal{D}_0+\frac{\gamma q}{\eta}),\frac{Bq\gamma_1}{\mu^2\eta},\frac{Bq\gamma_2}{\mu^2\eta},\frac{q}{\eta},1\Big)^T$ 
it holds that $$ 
\mathcal{N}(D_{\bf E}F(\mathcal{D}_0,{\bf E}^0))=span\{ {\bf J}^0(\mathcal{D}_0)\} \ \text{and}\ \mathcal{R}(D_{\bf E}F(\mathcal{D}_0,{\bf E}^0))=\{\bf{E}\in\mathbb{R}^5\ ;\ <{\bf E},{\bf e}_4>=0\}. 
$$
 Thus $\mathcal{K}{\bf J}^0(\mathcal{D}_0)=(-\frac{B}{\mu},0,0,0,\frac{B}{\mu})^T\notin \mathcal{R}(D_{\bf E}F(\mathcal{D}_0,{\bf E}^0))$.

 $(ii)$ It can also be verified by inspection.
\end{proof}
Next, using  Lemma \ref{lem-3-2}, Lemma \ref{lem-1-0-0-0}, and the theory of bifurcation from simple eigenvalues \cite{Crandall1,crandall2}, we can now present the proof of Theorem \ref{Th-5}.

\begin{proof}[Proof of Theorem \ref{Th-5}]
Let us suppose that all the parameters $q,\eta,\mu,d,\gamma_1,\gamma_2,\gamma_0>0$ and $\beta_1,\beta_2\in(0,1)$ are fixed and emphasize only on  the dependence of the endemic equilibrium ${\bf E}^*(\beta_0)$ with respect to $\beta_0>\mathcal{D}_0-\frac{\mu}{B}\min\{\mu,\eta\}$. Thanks to Lemma \ref{lem-3-2} $(i)$ and \cite[Theorem 1.7]{crandall2}, there exist some $\varepsilon>0$ and smooth functions 
$$ 
\varphi \ : \ (-\varepsilon, \varepsilon) \to \mathbb{R} \ \ \text{with}\ \varphi(0)=0
$$
and 
$$ 
\Phi \ : \ (-\varepsilon,\varepsilon)\ \to span\{{\bf J}^0(\mathcal{D}_0)\}^T \quad \text{with}\ \Phi(0)=0
$$
such that $ F(\beta_0(s),{\bf E}(s))=0$ for every $|s|<\varepsilon$, where 
\begin{equation}\label{last 1}
\beta_0(s)=\mathcal{D}_0+\varphi(s)\quad \text{and}\quad {\bf E}(s)={\bf E}^0+s{\bf J}^0(\mathcal{D}_0)+s\Phi(s)
\end{equation}
with $span\{{\bf J}^0(\mathcal{D}_0)\}^T$ is the orthogonal complement of $ span\{{\bf J}^0(\mathcal{D}_0)\}$. Moreover Lemma \ref{lem-3-2} $(ii)$ and \cite[Corollary 1.13 \& Theorem 1.16]{Crandall1} imply that there exist smooth functions $r^* : (-\varepsilon, \varepsilon)\to \mathbb{R}$ and ${\bf J}^* \ :\ (-\varepsilon,\varepsilon) \to \mathbb{R}^5 $ satisfying 
\begin{align*}
& (i)\quad  D_{{\bf E}}F(\beta_0(s),{\bf E}(s))({\bf J}^*(s))=r^*(s){\bf J}^*(s)\quad \forall\ s\in(-\varepsilon,\varepsilon)\cr
& (ii) \quad {\bf J}^*(0)={\bf J}^0(\mathcal{D}_0) \ \text{and}\ r^*(0)=r(\mathcal{D}_0)=0\cr
& (iii) \quad r^*(s)\ \text{and}\ -s\varphi'(s)r'(\mathcal{D}_0) \ \text{have the same zeros and, same sign whenever } r'(\mathcal{D}_0)\ne0,
\end{align*}
where $r'(\cdot)$ is the derivative of $r(\cdot)$.
From the expression $r(\beta_0)=\frac{B}{\mu}(\beta_0-\mathcal{D}_0)$ it easily follows that $r'(\beta_0)=\frac{B}{\mu}$ for every $\beta_0$. Thus it follows from $(iii)$ that $r^*(s)$ has the same sign as $-s\varphi'(s)$ for $|s|$ small.  Next, we determine the sign of $\varphi'(s)$ for $|s|\ll 1$.  Observe from \eqref{last 1} that 
$$ 
\frac{1}{s}<{\bf E}(s),{\bf e}_4>=1+<\Phi(s),{\bf e}_4>\ \longrightarrow 1\  \text{as}\ s\to 0.
$$ This shows that $I'(0)=1$ and hence ${\bf E}(s)$ is an equilibrium solution of \eqref{e-0-1} for which $I(s)\ne 0$ for $0<|s|\ll 1$. As a result, recalling the function $G$ introduced in \eqref{g-defi}, we must have that $G(I(s),\beta_0(s),\frac{q}{\eta})=\frac{\mathcal{D}_0B}{\mu}$ for $0<|s|\ll 1$. This is equivalent to saying that 
$$ 
\left(\mu+\left(\beta_0(s)+\frac{q}{\eta}\gamma\right)I(s)\right)\frac{\mathcal{D}_0}{\mu}-\beta_0(s)= \frac{q}{\eta}\sum_{j=1}^2\frac{\beta_j\gamma_jI(s)}{\mu+\beta_jI(s)} \quad \forall |s|\ll 1.
$$
Taking the derivative of both sides with respect to $s$ yields 
\begin{equation}\label{last-2} 
\frac{\mathcal{D}_0}{\mu}\left( \beta_0'(s) I(s) +\left(\beta_0(s)+\frac{q}{\eta}\gamma\right)I'(s)\right)-\beta_0'(s)=\frac{\mu qI'(s)}{\eta}\sum_{j=1}^2\frac{\beta_j\gamma_j}{(\mu+\beta_jI(s))^2}.
\end{equation}
Evaluating this equation at $s=0$ and using the fact that  $\beta_0(0)=\mathcal{D}_0$, $I(0)=0$ and $I'(s)=1$ yield
$$ 
\varphi'(0)=\beta_0'(0)=\frac{1}{\mu}\left(\mathcal{D}_0\left( \mathcal{D}_0+\frac{q}{\eta}\gamma\right)-\frac{q}{\eta}\sum_{j=1}^2\beta_j\gamma_j\right).
$$
From this point, we distinguish two cases.\\
{\bf Case 1}. $\varphi'(0)>0$, so that  there is a transcritical bifurcation  at $\beta_0=\mathcal{D}_0$.  Then  $\beta_0(s)>\mathcal{D}_0$ for $0<s\ll 1$ and ${\bf E}(s)$ is an endemic equilibrium of \eqref{e-0-1} corresponding to $\beta_0=\beta_0(s)$ for $0<s\ll 1$. Using $(iii)$ we conclude that $r^*(s)<0$ and $\varphi'(s)>0$ for every $0<s\ll 1$. It is clear from $(i)$ that $r^*(s)$ is an eigenvalue of $D_{{\bf E}}F(\beta_0(s),{\bf E}(s))$. Hence, recalling  that the eigenvalues of $ D_{{\bf E}}F(\beta_0,{\bf E}^0(\beta_0))$ consist of the set $\{-\mu,-\eta,r(\beta_0)\}$ with $r(\beta_0)>0$ as the largest eigenvalue, we conclude that there is some $0<\varepsilon^*<\varepsilon$ such that $r^*(s)$ is also the largest eigenvalue of $D_{{\bf E}}F(\beta_0(s),{\bf E}(s))$ for $s\in (-\varepsilon^*,\varepsilon^*)$. Therefore, ${\bf E}(\beta_0(s))$ is linearly stable for every $s\in (0,\varepsilon^*)$. Now, recalling Theorem \ref{lem-1-0-0-0}, we know that the only endemic equilibrium  of \eqref{e-0-1} corresponding to $ \beta_0(s)$ is ${\bf E}^*(\beta_0(s))$. Thus we have that ${\bf E}^*(\beta_0(s))={\bf E}(s)$ and  $s\in (0,\varepsilon^*)$. Since the function $(0,\varepsilon^*)\ni s\mapsto \beta_0(s)$ is strictly increasing and hence invertible, then for every $\beta_0\in(\mathcal{D}_0,\beta_0(\varepsilon^*))$, we have that ${\bf E}^*(\beta_0)$  is linearly stable. 
\\
{\bf Case 2.} $\varphi'(0)=0$, that is, we have a pitchfork bifurcation at $\beta_0=\mathcal{D}_0$. Differentiating again \eqref{last-2} with respect to $s$ gives 
\begin{align*}
&\frac{\mathcal{D}_0}{\mu}\left(\beta_0''(s)I(s)+2\beta_0'(s)I'(s)+\left(\beta_0(s)+\frac{q}{\eta}\gamma\right)I''(s)\right)-\beta_0''(s)\cr
=&\frac{\mu q I''(s)}{\eta}\sum_{j=1}^2\frac{\beta_j\gamma_j}{(\mu+\beta_jI)^2}-2\frac{\mu q[I'(s)]^2}{\eta}\sum_{j=1}^2\frac{\beta_j^2\gamma_j}{(\mu+\beta_jI(s))^3}.
\end{align*}
Evaluating the last equation at $s=0$ and using the fact that $\beta'(0)=0$, i.e, $\mathcal{D}_0\big( \mathcal{D}_0+\frac{q}{\eta}\gamma\big)=\frac{q}{\eta}\sum_{j=1}^2\beta_j\gamma_j$, we obtain 
$$ 
\beta_0''(s)=2\frac{q}{\eta \mu^2}\sum_{j=1}^2\beta_j^2\gamma_j>0.
$$
Thus we obtain that $\varphi'(s)>0$ for $0<s\ll 1$, so we are back to {\bf case 1} and we can employ the same arguments to deduce that ${\bf E}(\beta_0)$ is linearly stable for $\beta_0\in(\mathcal{D}_0,\beta_0(\varepsilon^*))$.

In either case  we may now take  $$\beta_{0,\max}:=\sup\{\tilde{\beta}>\mathcal{D}_0\ :\ \max\{Re(\lambda)\ :\ \lambda\in\sigma(D_{{\bf E}}F(\beta_0,{\bf E}(\beta_0)))\} <0 \ \ \forall\ \mathcal{D}_0<\beta_0<\tilde{\beta}_0  \},$$
where $ \sigma(D_{{\bf E}}F(\beta_0,{\bf E}(\beta_0)))$ is the spectrum of $D_{{\bf E}}F(\beta_0,{\bf E}(\beta_0))$  and $Re(\lambda)$ is the real part of the complex number $\lambda$. This completes the proof of Theorem \ref{Th-5}.

\end{proof}
\section{Permanence of disease  when $\beta_{0,u}>\mathcal{D}_0$.}\label{Sec-for-persistence}
In this section,  we discuss the permanence of the disease when $\beta_{0,u}>\mathcal{D}_0$. Recall again that we have set  $\beta_{j}=\beta_{j,u}$ for each $j=0,1,2$ in order to simplify notations. To this end, we first prove the following lemma.
 
 \begin{lem}[Uniform weak permanence]\label{weak-persistence} Let $u\ge 0$ be fixed and
 suppose that $\beta_0>\mathcal{D}_0$. Then 
 for every solution  $(S_0(t),S_1(t),S_2(t),Z(t),I(t))$ of \eqref{e-0-1} with initial in  $\mathcal{X}_u$ satisfying $I(s)>0$ for every $s\in[-u,0]$,
 it holds that
 \begin{equation}\label{l-2}
 \overline{I}:=\limsup_{t\to\infty}I(t)\geq \underline{m}:= \frac{\mu(\beta_0-\mathcal{D}_0)}{\mathcal{D}_0(\beta_0+\frac{q\gamma}{\eta})} >0.
 \end{equation}
 
 \end{lem}
\begin{proof}
We proceed by contradiction to establish the result. So, we suppose that there is a positive  solution  $(S_0(t),S_1(t),S_2(t),Z(t),I(t))$ of \eqref{e-0-1} with initial in  $\mathcal{X}_u$ satisfying $I(s)>0$ for every $s\in[-u,0]$ and $\lambda_1\in(0,1)$ such that 
\begin{equation}\label{l-3} 
  \overline{I}:=\limsup_{t\to\infty}I(t)< \lambda_1\underline{m}.
\end{equation}
We fix $\lambda_2\in(0,1-\lambda_1)$. Using the definition of limsup, there is $t_1\gg 1$ such that 
\begin{equation}\label{l-4} 
\sup_{t\geq t_1}I(t)\le (\lambda_1+\lambda_2)\underline{m}.
\end{equation}
Recalling the equation satisfied by $\frac{d}{dt}Z(t)$, we obtain that 
\begin{equation*}\label{l-4-1}
    \frac{d}{dt}Z(t)\leq (\lambda_1+\lambda_2)q\underline{m}-\eta Z(t)\quad \forall\ t\geq t_1.
\end{equation*}
Thus, if we fix $\lambda_3\in(0,1-\lambda_1+\lambda_2)$,
there is some $t_{2}>t_1+u$ such that \begin{equation}\label{l-4-2}
    Z(t)\leq \frac{(\lambda_1+\lambda_2+\lambda_3)q\underline{m}}{\eta}\quad  \forall\ t\geq t_{2}.
\end{equation}
It then follows from \eqref{l-4} and \eqref{l-4-2} that 
$$ 
\frac{d}{dt}S_{0}(t)\geq B-\left(\mu+\left((\lambda_1+\lambda_2)\beta_0+\frac{(\lambda_1+\lambda_2+\lambda_3)q\gamma}{\eta}\right)\underline{m}\right)S_{0}, \quad\ \forall\ t\geq t_{2}.
$$
As a result, 
there is $t_{3}>t_{2}$ such that 
\begin{equation}\label{l-4-3}
    S_{0}(t)\geq \frac{B}{\mu+(\lambda_1+\lambda_2+\lambda_3)\left(\beta_0+\frac{q\gamma}{\eta}\right)\underline{m}}\quad \forall\ t\geq t_{3}.
\end{equation}
This in turn implies that 
\begin{equation}\label{l-4-5}
\frac{d}{dt}I(t)\geq \left(\frac{\beta_0B}{\mu+(\lambda_1+\lambda_2+\lambda_3)\left(\beta_0+\frac{q\gamma}{\eta}\right)\underline{m}}\right)I(t-u)-\frac{B\mathcal{D}_0}{\mu}I\quad t\geq t_{3}.
\end{equation}
Observe from the formula of $\left(\beta_0+\frac{q\gamma}{\eta}\right)\underline{m}=\frac{\mu({\beta}_0-\mathcal{D}_0)}{\mathcal{D}_0} $  and the fact that $\lambda_1+\lambda_2+\lambda_3<1 $ that 
\begin{align*}
\frac{\beta_0B\mu}{\left(\mu+(\lambda_1+\lambda_2+\lambda_3)\left(\beta_0+\frac{q\gamma}{\eta}\right)\underline{m}\right)B\mathcal{D}_0}>&\frac{\beta_0 B\mu}{ \left(\mu+\left(\beta_0+\frac{q\gamma}{\eta}\right)\underline{m}\right)\mathcal{D}_0B}\cr
=&\frac{\beta_0B \mu}{\left(\mu +\frac{\mu(\beta_0-\mathcal{D}_0)}{\mathcal{D}_0}\right)B\mathcal{D}_0}=1.
\end{align*}
Thus, similar arguments as in {\bf Case 3} of the proof of Lemma \ref{lem-1-0-01} show that the algebraic equation
$$ 
\lambda- \left(\frac{\beta_0B}{\mu+(\lambda_1+\lambda_2+\lambda_3)\left(\beta_0+\frac{q\gamma}{\eta}\right)\underline{m}}\right)e^{-\lambda u}+\frac{\mathcal{D}_0B}{\mu}=0
$$  
has a unique positive root $\lambda_0>0$. 
Hence, 
the solution $\underline{I}(t)$ to the first-order linear delay differential equation
\begin{equation} \label{l-4-4}
\begin{cases}
\frac{d}{dt}\underline{I}=\left(\frac{\beta_0B}{\mu+(\lambda_1+\lambda_2+\lambda_3)\left(\beta_0+\frac{q\gamma}{\eta}\right)\underline{m}}\right)\underline{I}(t-u)-(\mu+d)\underline{I} & t>t_{3},\cr
\underline{I}(s)=I(s)>0 & s\in[t_3-u,t_3],
\end{cases}
\end{equation}
blows up exponentially as $t\to\infty$. From \eqref{l-4-5} and \eqref{l-4-4}, we conclude that 
$$ 
I(t)\ge \underline{I}(t)\quad \forall\ t\geq t_{3}. 
$$
As a result, we obtain that 
$$ 
\limsup_{t\to\infty}I(t)\ge \lim_{t\to\infty}\underline{I}(t)=\infty.
$$

This clearly contradicts \eqref{l-3}. Therefore, the statement of the Lemma must hold.
\end{proof}
We first remark that both inequalities \eqref{amz-1} and \eqref{l-2} complete the proof of \eqref{persistent-eq-1}. Thanks to Lemma \ref{weak-persistence} and the theory developed in \cite{Hal}, we can now conclude our next result on the uniform persistence of the disease whenever $\beta_0>\mathcal{D}_0$. We first introduce few terminologies.  Note that \eqref{e-0-1} generates a continuous  semiflow on the set 
$$ 
\tilde{\mathcal{X}}_u:=\{(S_0,S_1,S_2,Z_0,I_0(\cdot)) \in\mathcal{X}_u\ :\ I(s)>0\ \forall\ s\in[-u,0] \ \text{and}\ S_0>0\},
$$
which we denote by $\Phi_t$, i.e, $\Phi_t(S_0,S_1,S_2,Z_0,I_0(\cdot))=(S_0(t),S_1(t),S_2(t),Z(t),I_t(\cdot))$ for every $t\ge 0$ with $I_t(s):=I(t+s)$ for every $s\in[-u,0]$ and $t\ge 0$. We note from the Arzela-Ascoli's Theorem that $\Phi_{t}$ is compact for every $t>u$. For every $(S_0,S_1,S_2,Z_0,I_0(\cdot)) \in\tilde{\mathcal{X}}_u$, define the persistence function
\begin{equation}
    \rho((S_0,S_1,S_2,Z_0,I_0(\cdot)))=I_0(0) \quad \forall\ (S_0,S_1,S_2,Z_0,I_0(\cdot))\in\tilde{\mathcal{X}}_u.
\end{equation}
In particular,
$$ 
\rho(\Phi_t(S_0,S_1,S_2,Z_0,I_0(\cdot)))=I_t(0)=I(t) \quad \forall\ t\geq 0 \ \text{and}\ (S_0,S_1,S_2,Z_0,I_0(\cdot))\in\tilde{\mathcal{X}}_u.
$$
Recall that $I(t)>0$ for every $t>0$ whenever $I_0(0)>0$. Hence 
$$ 
\rho(\Phi_t(S_0,S_1,S_2,Z_0,I_0(\cdot)))>0 \quad \text{whenever}\ \rho(\Phi_0(S_0,S_1,S_2,Z_0,I_0(\cdot)))>0.
$$

\begin{proof}[Proof of Theorem \ref{tm-uniform-persistency}]
Recall that $ \Phi_t$ is compact for every $t>u$. We also recall that 
$$ 
\sum_{j=0}^2|S_j(t)|+\|I_t\|_{\infty}\leq \frac{2B}{\mu} \quad \text{and}\quad |Z(t)|\leq \frac{qB}{\eta\mu} , \quad \forall\ t\ge u, 
$$
which implies that $\Phi$ is pointwise dissipative and uniformy bounded for $t\ge u$. Therefore, 
it follows from \cite[Theorem 2.30 Page 41]{Hal} that \eqref{e-0-1} has a compact attractor $\mathcal{A}$ of $\tilde{\mathcal{X}}_u$. Next, observe that the map $\rho\circ\Phi$ is continuous, and  Lemma \ref{weak-persistence} guarantee that $\Phi$ is uniformly weak $\rho-$ persistent. Therefore, it follows from \cite[Theorem 5.2 Page 126]{Hal} that $\phi$ is uniformly  $\rho-$persistent, i.e, there exists $\overline{m}_u>0$ such that 
$$ 
\liminf_{t\to\infty}\rho(\Phi_t(S_0,S_1,S_2,Z_0,I_0(\cdot)))>\overline{m}_u\quad \forall \, (S_0,S_1,S_2,Z_0,I_0(\cdot))\in \tilde{\mathcal{X}}_u.
$$
This completes the proof of the Theorem since $\rho(\Phi_t(S_0,S_1,S_2,I_0,Z_0))=I(t)$ for every $t>0$ and $(S_0,S_1,S_2,Z_0,I_0(\cdot))\in\tilde{\mathcal{X}}_u$.
\end{proof}

\section{Numerical investigations}\label{sec3}
In this section we consider investigating the numerical studies related to the theoretical results outlined in previous sections.  We will consider  model~\eqref{e-0} with $u=3k$ for $k=0,1,2,3,4$. 

\subsection{Model and Data}\label{without_delay}

While any set of data related to compartment classes in our model could be used to illustrate numerical simulations, we decide to use the data from Uganda within the period 1992-2005.  Using such a real data  has the advantage of testing our mathematical model, and we can  estimate the different infection rates with respect to the time delay. As a result, we are able to confirm that the infection rates decrease as functions of the time delay. The data is reported in \cite{joshietal2008} and is given in Table \ref{tab:uganda_data}.
\begin{table}[htbp]
\caption{Data table}
    \centering
    \begin{tabular}{cccc}
    \hline
     Year & Susceptible & Infected & Information \\\hline
     1997 &             &          &  600/1,200    \\
     1999 & 6,700,000   & 606,000  &               \\
     2000 & 6,597,470   & 573,693  &  700/1,200    \\
     2001 & 7,130,000   & 383,000  &  717/1,200    \\
     2003 & 7,462,000   & 544,000  &               \\
     2005 & 7,636,000   & 548,261  &  778/1,200    \\\hline
    \end{tabular}
    \label{tab:uganda_data}
\end{table}

We recall that the data in Table~\ref{tab:uganda_data} was collected from UN, UNAIDS, and the Uganda AIDS Commission for population, which included death rates, percentage of adults (ages 15 to 59) per year, growth of the adult class, the adult prevalence rates, and the percentage of adult population infected. Moreover, subject matter experts, literature, and surveys were used to determine organizational estimates of the information and education campaigns involved, rates for interaction of $Z$ to susceptible classes, and efficacy of abstinence and faithfulness and the use of condom behavior. We emphasize that our model and the data collected considers only adult individuals ages 15 to 59.

It is important to note that our model~\eqref{e-0} differs from the model considered by the authors in~\cite{joshietal2008} in the sense that in~\cite{joshietal2008}, the logistic growth was used for the $Z$ class with a growth coefficient that increases with the number of infectives and is multiplied by a ratio of the infected to the living population. Our model assumes that the growth of information is proportional to the number of infectives. Moreover, the $R$ class in~\cite{joshietal2008} only collects the number of people that die from the disease, while in our case the $R$ population represents the number of susceptibles that will not contract the disease.  Furthermore, the model in~\cite{joshietal2008} is a system of ordinary differential equations while our model is a system of delay differential equations. Finally, no mathematical analysis was provided in \cite{joshietal2008}.

The authors in~\cite{joshietal2008} calculated values of the parameters $B$, $\gamma_1$, $\gamma_2$, $d$, and the general or natural death rate $\mu$. Also, it was assumed that $\beta_{1,u}$ and $\beta_{2,u}$ are directly proportional to $\beta_{0,u}$, that is, $\beta_{1,u} = 0.05\beta_{0,u}$ and $\beta_{2,u} = 0.40\beta_{0,u}$. They then use parameter estimation to fit the model to the data in order to estimate together $\beta_{0,u}$ and the logistic growth rate for the $Z$ class.

\subsection{Parameter estimation}\label{estimation}
Since the data concerns only adult individuals ages 15 to 59, we use $\mu = \frac{1}{59-15} = 0.0227$ in our model. We choose to fix all  remaining parameters as outlined in the previous section but we decide to fit the model~\eqref{e-0} to the data in Table~\ref{tab:uganda_data} in order to estimate the parameters $\beta_{0,u}$, $\eta$, $\gamma_0$, and $q$. We recall the values of the fixed parameters in  Table~\ref{tab:fixed_parameters}.  

\begin{table}[htbp]
\caption{Values of fixed parameters in model~\eqref{e-0}}
    \centering
    \begin{tabular}{ccccccc}\hline
    B  &  $ \beta_{1,u} $ & $\beta_{2,u}$  & $d$ & $ \gamma_1$ & $ \gamma_2$ &  $\mu$ \\\hline
    0.55 &  $0.05\beta_{0,u}$ & $0.40\beta_{0,u}$ & 0.14  & 0.1  & 0.8    & 0.0227 \\\hline   
    \end{tabular}
    \label{tab:fixed_parameters}
\end{table}

In order to fit the model to the data, we have to determine the initial conditions for the system of differential equations. Now, it is known that the information and educational campaigns against HIV/AIDS in Uganda started in 1992. For this reason, we will start our model with 1992. We will assume that initially no one in the population is practicing abstinence and faithfulness or the use of condom behavior, or any type of information related to HIV/AIDS. This means that we have $S_1(0) = 0$, $S_2(0) = 0$, and $R(0) = 0$. 

According to the UN population reports, there were 18.38 million individuals in Uganda in 1992, of which $32.1\%$ represented the adult population ages 15 to 59. This allows us to determine that there were 5,899,980 adult individuals for ($S_0(0) + I(0)$) class for that year. At the same time, it is also reported that the adult prevalence rate was about $15\%$, giving us 884,997 infected adult individuals initially. That is, $I(0) = 884,997$. Thus, we conclude that there were 5,014,983 susceptible adult individuals initially for that year. That is, $S_0(0) = 5,014,983$. For $Z(0)$, the authors in~\cite{joshietal2008} estimated that initially there were 240 organizations out of 1,200 (that is, $20\%$) that were involved in the abstinence and faithfulness, and the use of condom behaviors. Hence, $Z(0) = 0.20$.

Next, it is important to choose the unit under which we will run the numerical simulations. Indeed, if population numbers are too large or too small, the rounding error may affect the efficiency of the results. To this end, we choose to work with the unit of a millionth. We can summarized the initial conditions as follows
\begin{equation*}
(S_0(0), S_1(0),S_2(0),I(0),Z(0),R(0))^T =( 5.014983, 0,0, 0.884997, 0.20, 0)^T.
\end{equation*}

We use MATLAB to run the numerical simulations. We develop and implement an algorithm to fit our model to the data. Specifically, for the delay differential system~\eqref{e-0}, our program creates three nested functions within a main function. One nested function specifies the history of the solution $(S_0(t), S_1(t), S_2(t), I(t), Z(t), R(t))^T$ for $t\le 0$. In our case we use the initial conditions for this history. The other nested function implements the delay system to evaluate the right side of the differential equations~\eqref{e-0}. Finally, a third nested function uses the delay differential equations (DDE) solver {\it dde23} to solve the system~\eqref{e-0}, then calculates and returns the relative sum of squared errors (SSE) between the data in Table~\ref{tab:uganda_data} and the model solution~\eqref{e-0}.

On the other hand, we will also be interested in the solution of system~\eqref{e-0} 
  with $u=0$, i.e., there is no delay. In this case, our program includes two nested functions within a main function. One nested function uses the ODE solver \textit{ode45} to implement and solve the system~\eqref{e-0} with $u = 0$. The other nested function receives as input arguments the parameters to be estimated and time $t$, then calculates and returns the relative sum of squared errors (SSE) between the data in Table~\ref{tab:uganda_data} and the model solution~\eqref{e-0} with $u = 0$. The relative SSE is calculated as follows:
\begin{align*}
    \text{Model error $(\beta_{0,0},\eta,\gamma_0,q)$} &= \sum_j\left[\frac{|\text{Model$(S_0 + S_1 + S_2)_j$} - \text{Data$(S)_j$}|^2}{|\text{Data$(S)_j$}|} + \frac{|\text{Model$(I)_j$} - \text{Data$(I)_j$}|^2}{|\text{Data$(I)_j$}|}\right] \\
    & + \sum_k\left[ \frac{|\text{Model$(Z)_k$} - \text{Data$(Z)_k$}|^2}{|\text{Data$(Z)_k$}|}\right], 
\end{align*}
where $j = 1999, 2000, 2001, 2003, 2005$ and $k = 1997, 2000, 2001, 2005$.

Moreover, within the main functions for both the model~\eqref{e-0} with $u>0$ and the same model with $u=0$, we use the MATLAB nonlinear programming solver \textit{fminsearch}, which finds minimum of unconstrained multi-variable function using derivative-free method. In our case, \textit{fminsearch} receives two input arguments, a handle to the nested error function and the initial guesses of the parameters to be estimated, then it returns the values of the fitted parameters. 

We also emphasize that in all of our simulations, we use mass action incidence terms like $\beta_{0,u}S_0(t)I(t)$. The standard incidence terms like $\beta_{0,u}S_0(t)I(t)/N$, where $N$ is the total population didn't seem to provide a good fit between data and the model solution.

Now, obtaining pre-estimates or initial guesses of the parameters to be determined is not an easy task. In order to obtain ``good" pre-estimates, we reached out to Wolfram Mathematica's \textit{Manipulate} command. This command allows, after setting up the model~\eqref{e-0} with quantities to be estimated treated as parameters, to plot both the data and the model solutions on the same graph. Then ``manipulate" them until an agreement between the data and the model solutions is achieved. We then select the parameter values for which this agreement is made and take them as initial guesses. 




\subsection{Numerical Results}\label{numerics}
We will now investigate the numerical behavior of the model~\eqref{e-0}. 
To demonstrate the importance of the information $Z$, we will first describe model~\eqref{e-0} without such information, i.e., with $Z = 0$. In this case the system reduces into the simplest HIV epidemic model, the following SI model 
\begin{equation}\label{no-Z}
\begin{cases}
\frac{d}{dt} S_{0}(t)=  B -\beta_{0,u} S_{0}(t) I(t-u)  -\mu S_{0}(t),\cr
\frac{d}{dt} I(t)=\beta_{0,u} S_{0}(t) I(t-u)  -(\mu + d) I(t).
\end{cases}
\end{equation}

It is important to note that the model~\eqref{no-Z} does not fit HIV data well. This could possibly be explained from the fact that for a simple SI model, the time spent in the infectious stage decreases exponentially, which is an unrealistic assumption for HIV, for which the infectious stage is long and the duration is affected by several variations. See for instance,~\cite[page 131]{Maia2015}. For this reason we will not fit the model~\eqref{no-Z} to the data. However, for the purpose of illustration to show how useful information campaigns are, we will provide numerical solutions to model~\eqref{no-Z}. To this end, we use parameter values for $B$, $d$, and $\mu$ in Table~\ref{tab:fixed_parameters}. For the choice of $\beta_{0,u}$, we have to make sure that neither the $I$ class nor the $S_0$ class fits their respective data. For if this happens, it will result in favoring one class over the other. Thus, $\beta_{0,u}= 0.023549$ as in Table~\ref{tab:parameterDelay6} is an interesting choice. Also, we choose to run the model with the time delay $u = 6$. The graphs of the two classes of the model together with their data are given in the first two figure windows in Figure~\ref{fig:no_Z}. The last figure window provides only the graphs of the two classes of the model~\eqref{no-Z} plotted together.
\begin{figure}[hbtp]    
\centering    
\includegraphics[width=1\textwidth,height=.25\textheight]{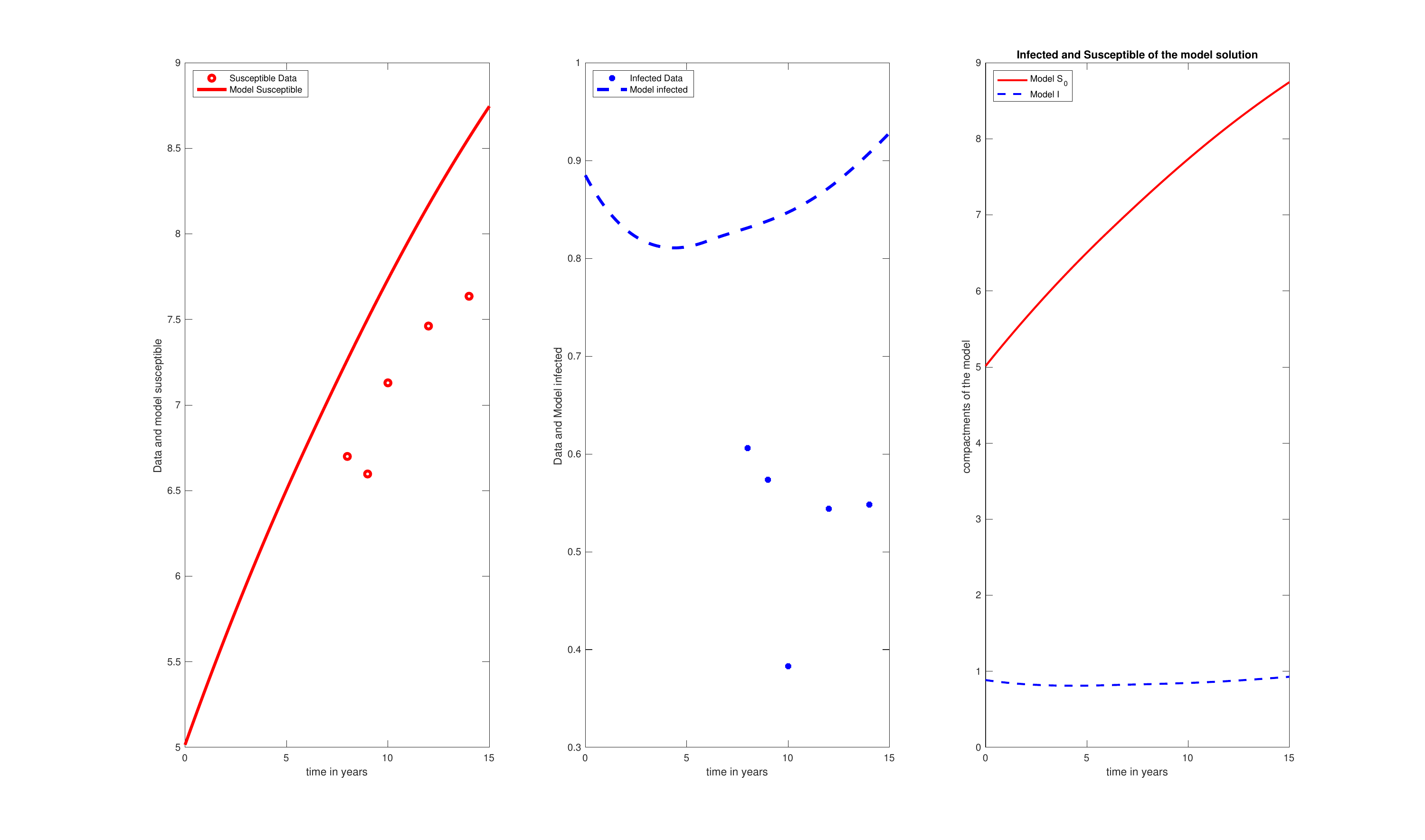}
\caption{The graphs of the infected and susceptible model solutions and their data for model~\eqref{no-Z} calculated with parameter values from Table~\ref{tab:fixed_parameters} and $\beta_{0,6}$ in Table~\ref{tab:parameterDelay6}.}   
\label{fig:no_Z}
\end{figure}
Clearly, we can observe that in the absence of information, more and more people joins the at-risk population, and as a consequence, the number of infected individuals increases.

 Now, we will fit the model~\eqref{e-0} to the data in Table~\ref{tab:uganda_data} for one time delay value $u = 6$ and obtain the resulting fitted model solutions over the interval $[0,15]$ representing the years from 1992 to 2006. We use the Wolfram Mathematica's \textit{Manipulate} command to obtain the initial guesses of the parameters. The reader interested in this valuable process can learn from our code and the accompanying graph in Figure~\ref{fig:fithivddeDelay6}.
\begin{figure}[hbtp]    
\centering    
\includegraphics[width=.85\textwidth,height=.3\textheight]{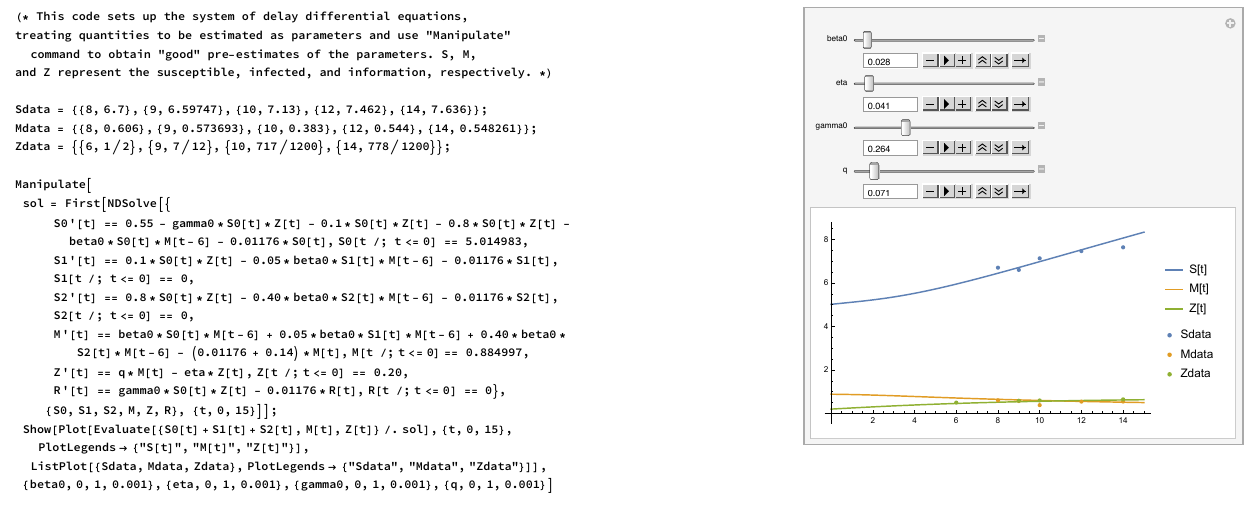}
\caption{Wolfram Mathematica code using the {\it Manipulate} Command to pre-estimate parameters with delay $u = 6$ in model~\eqref{e-0}.}   
\label{fig:fithivddeDelay6}
\end{figure}

We obtain the initial guesses $\beta_{0,6} = 0.028$, $\eta = 0.041$, $\gamma_0 = 0.264$, and $q = 0.071$.
After running our program with these initial guesses, we obtain the estimates and their resulting SSE error given in Table~\ref{tab:parameterDelay6}.
\begin{table}[htbp]
    \centering
    \caption{Estimated parameters with corresponding SSE error for delay $u = 6$}
    \begin{tabular}{cccccc}\hline
    Delay $u$ & $\beta_{0,6}$ & $\eta$   & $\gamma_0$ & $q$      & SSE error \\\hline
    6 & 0.023549  & 0.088131 & 0.141355  & 0.128940   & 1.195502 \\\hline
    \end{tabular}
    \label{tab:parameterDelay6}
\end{table}
Figure~\ref{fig:fitalldelay6} shows the graphs of the fit between the data and the model solutions with $u = 6$, where, by the total susceptible population we mean the sum $S_0 + S_1 + S_2$.
\begin{figure}[hbtp] 
\centering    
\includegraphics[width=.8\textwidth,height=.3\textheight]{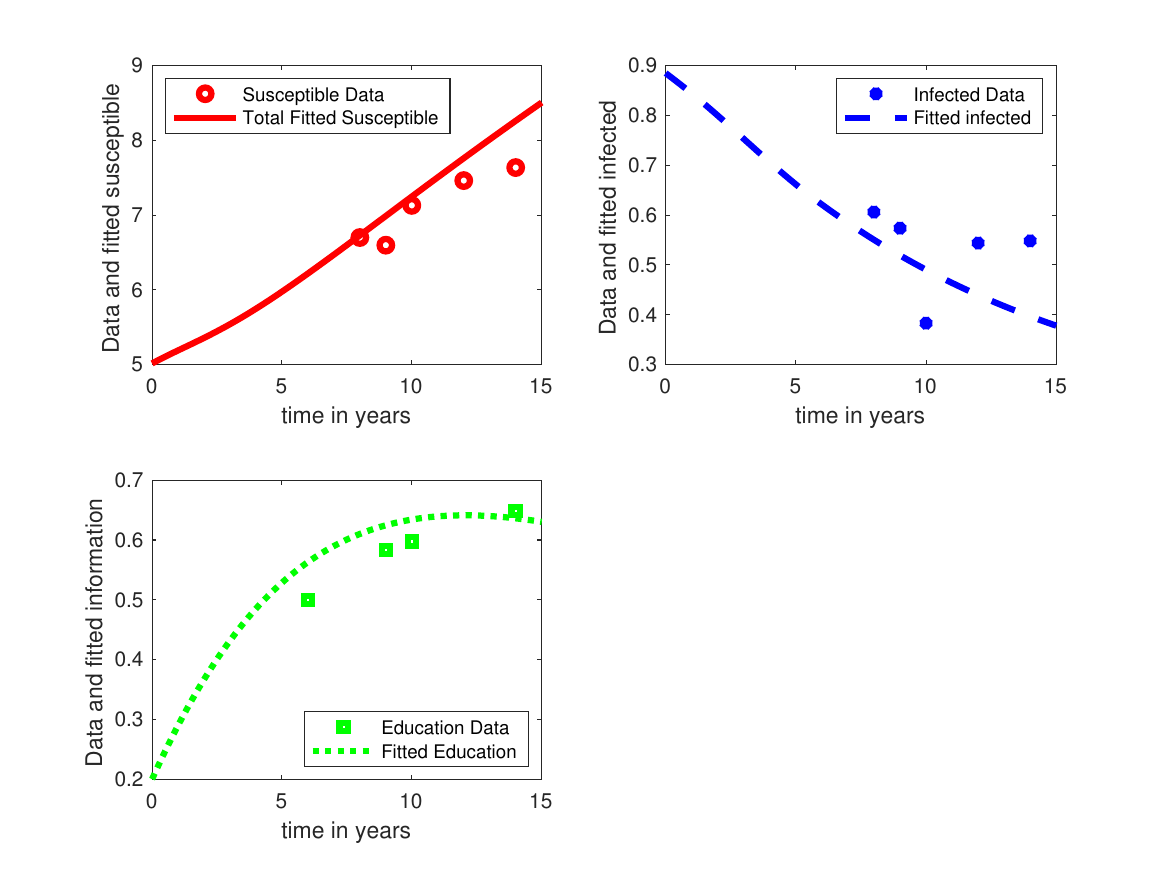}
\caption{Plots of the graphs of the fit between the model solutions and the data with time delay $u = 6$ and total susceptible population $S_0+S_1+S_2$.}
\label{fig:fitalldelay6} 
\end{figure}

To see the effect of information campaigns on the control of the disease, we also plot in Figure~\ref{fig:allgroupsDelay6} the graphs of the solutions for all subgroups of the model obtained with time delay $u = 6$. Note that this choice of $u=6$ is arbitrary and it has no biological meaning. In fact the simulations are similar when using different time delay. 
\begin{figure}[hbtp]    
\centering    
\includegraphics[width=.7\textwidth,height=.3\textheight]{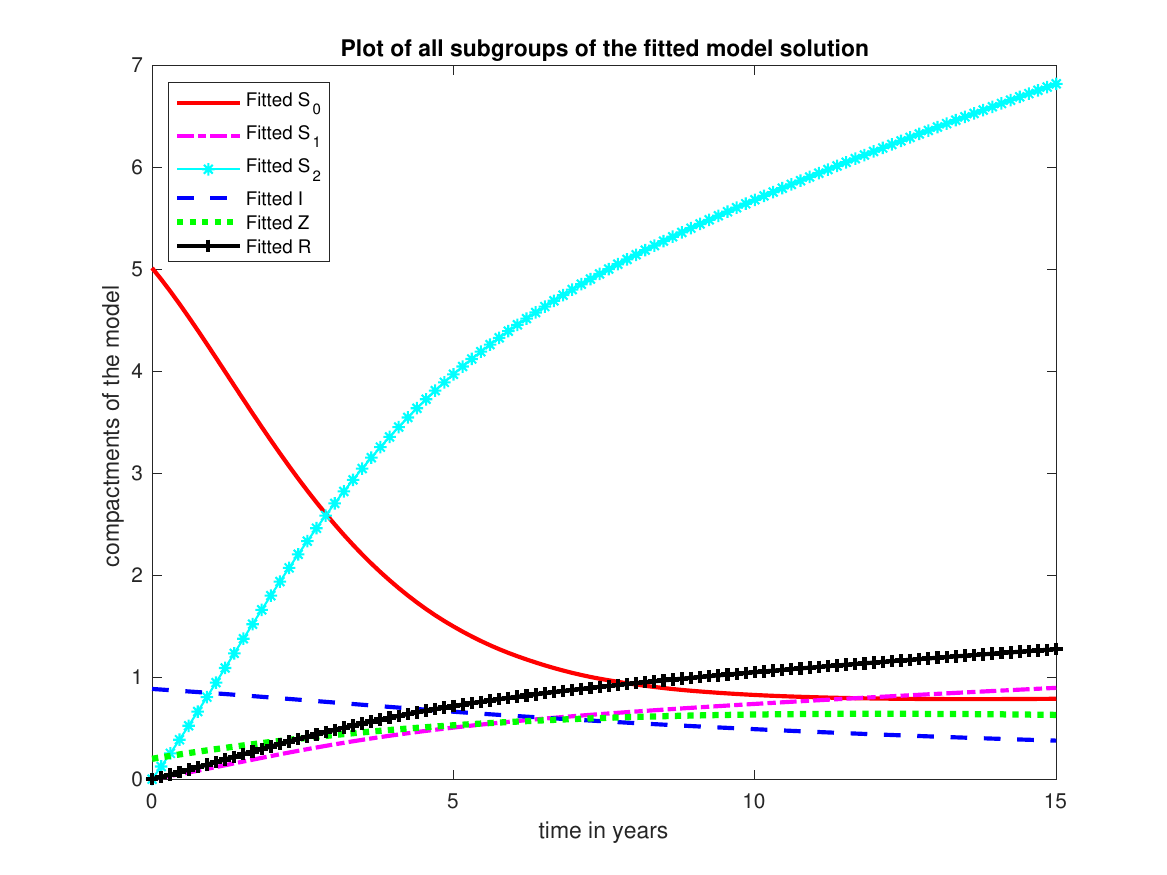}
\caption{The fitted model with each of its population subgroup obtained with delay $u = 6$.}   
\label{fig:allgroupsDelay6}
\end{figure}

The numerical results in Figure~\ref{fig:allgroupsDelay6} suggest that thanks to the information about the disease, the number $S_0$ of general susceptible people that has not yet changed their behavior before information decreases from 5 million to about a million once they are aware of the information. At the same time, once they acquire knowledge about the disease, 
 the numerical simulations show that more people will tend to use condoms than practicing abstinence and faithfulness. 
  As a consequence, the number of infected people decreases considerably. The fitted model here also shows that the number of removed people will surpass the number of susceptible practicing abstinence and faithfulness but will be smaller compared to the condom users. In light of this, it is clear that the information $Z$ plays a central role in decreasing the number of ignorant susceptible and the infected within this population in contrast to the numerical results from the last figure window in Figure~\ref{fig:no_Z} where there was no information education within the population.  

To understand the effect of the delay on the number of infected individuals, we use two approaches.
We will fix the estimated parameter values obtained by fitting the model to the data with $u = 6$ to obtain the model solutions for different other chosen time delay values and then look at the large time behavior of solutions.  So, in this first step we are assuming that the infection rates are constant with respect to the time delay.  Next, we  will fit the model for different time delays $u=3k$, $k=0,1,2,3,4$, and compare the infections rates $\beta_{0,u}$ obtained. This will help to check whether the rates decrease with the time delay. 

For the first approach, we fix the estimated parameter values obtained from fitting the model to the data over the interval $[0,15]$ with the time delay $u = 6$ as reported in Table~\ref{tab:parameterDelay6} and run the model solutions for the time delays $u = 0$, $u = 3$, $u = 6$, $u = 9$, and $u = 12$ with these estimated parameter values without doing another parameter estimation  over the time interval $[0,400]$. This guarantees that we have the same set of parameters for all chosen delays. We then compare the corresponding infected classes. The results are plotted in Figure~\ref{fig:bigtimeInfectedAllDelayNs}. 

\begin{figure}[hbtp]    
\centering    
\includegraphics[width=.7\textwidth,height=.3\textheight]{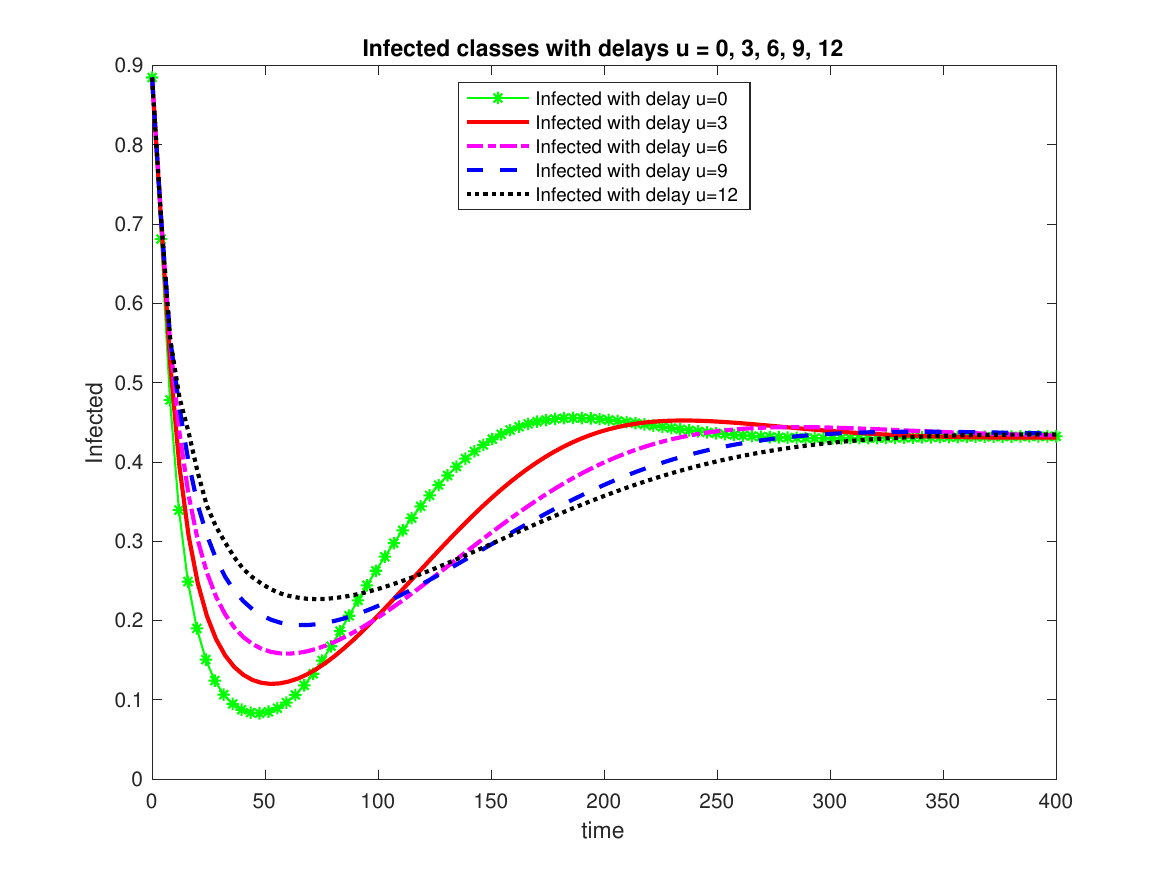}
\caption{The infected model solutions for the time delays $u = 0,3,6, 9,12$ over  $[0,400]$ obtained with estimated parameters from fitting the model to the data with delay $u = 6$ over $[0,15]$.}   
\label{fig:bigtimeInfectedAllDelayNs}
\end{figure}

 As noted earlier, the numerical simulations in Figure~\ref{fig:bigtimeInfectedAllDelayNs} are conducted with the same transmission rates $\beta_j=\beta_{j,u}, j=0,1,2$, for all delays. In other words, the endemic equilibrium of the system in this case is independent of the delay $u$. The simulations show that, irrespective of the time delay, the solutions converge to the unique endemic equilibrium in the long run. In particular all infected groups for all considered time delays converge to the infected component  of  the endemic equilibrium in the long run. Furthermore, it is clear from Figure~\ref{fig:bigtimeInfectedAllDelayNs} that the larger the delay the longer it takes for the infected population to stabilize at the endemic equilibrium. In other words, the infected group with smaller time delay converges faster to the unique infected component $I^*$ of the endemic equilibrium as given in Figure~\ref{fig:Istar}. This emphasizes one more time the importance of the time delay.

In the second approach, we try to understand the effect of the time delay on the infection rates and check whether our standing assumption on the non-increasing of rates with respect to the delay holds. We follow the procedures outlined in Section~\ref{without_delay} and Section~\ref{estimation} and fit the model to the data over the time interval $[0,15]$ for each delay $u=3k$, $k=0,1,2,3,4$. The initial guesses used for all delays are the same as reported in Figure~\ref{fig:fithivddeDelay6}. At the same time, we also record the effective rate of education dissemination $\tau = q/\eta$, the corresponding SSE error, and the corresponding basic reproductive number $\mathcal{R}_{0,u}$ in Table~\ref{tab:differentdelays}.
\begin{table}[htbp]
\caption{Estimated parameters with their corresponding errors for different delay values}
    \centering
    \begin{tabular}{cccccccc}\hline
    Delay $u$ & $\beta_{0,u}$ & $\eta$   & $\gamma_0$ & $q$      & $\tau$    & SSE error & $\mathcal{R}_{0,u}$ \\\hline
    0         & 0.028391      & 0.069813 & 0.138839   & 0.112060 & 1.605134  & 1.272334      & 4.222200 \\ \hline
    3         & 0.025545      & 0.079767 & 0.140867   & 0.120447 & 1.509985  & 1.239273      & 3.798930 \\\hline 
    6         & 0.023549      & 0.088131 & 0.141355   & 0.128940 & 1.463605  & 1.195502      & 3.502159 \\\hline
    9         & 0.022223      & 0.094835 & 0.141984   & 0.136189 & 1.435820  & 1.150862      & 3.304834 \\\hline
    12        & 0.021369      & 0.098370 & 0.142581   & 0.140660 & 1.429794  & 1.120674      & 3.177886 \\\hline 
    \end{tabular}
    \label{tab:differentdelays}
\end{table}

 As expected, we see from Table \ref{tab:differentdelays} that the infection rates decrease with respect to the time delay. Consequently, the basic reproductive number $\mathcal{R}_{0,u}$ and the infection rate $\beta_{0,u}$ decrease with $u$. Specifically, the numerical simulations suggest that it is possible to save more lives if there is a way to increase the time delay between the initial infection of cells by HIV and when a cell becomes infectious.
Furthermore, we expect that the effective education dissemination
rate $\tau:=\frac{q}{\eta}$ should decrease as the time delay increases. This is in fact confirmed by the estimates values for $\tau$ obtained in Table \ref{tab:differentdelays}. Therefore, the numerical results  illustrate our theoretical results.   

 We conclude this section with an illustration of Theorem~\ref{tm-existence-of-EE} on the existence and uniqueness of the endemic equilibrium, Theorem~\ref{prop-2} on the asymptotic behavior of $I^*$ with respect to $\tau = \frac{q}{\eta}$,  Theorem~\ref{prop-3} on the asymptotic behavior of $(S_0^*,S_1^*,S_2^*)$ with respect to $\tau=\frac{q}{\eta}$, and Theorem~\ref{tm-uniform-persistency} about the permanence of the disease. Using the values of the model parameters described in Table~\ref{tab:fixed_parameters} and in  Table~\ref{tab:parameterDelay6}, we have that
\[
\mathcal{D}_0 = \frac{\mu(\mu + d)}{B} = 0.0067, \quad \frac{1}{\gamma}\sum_{j=1}^2\beta_{j,u}\gamma_j = 0.0073.
\]
Clearly, we have that $\beta_{0,6}>\mathcal{D}_0$ and $\mathcal{D}_0<\frac{1}{\gamma}\sum_{j=1}^2\beta_{j,u}\gamma_j$. Therefore:
\begin{itemize}
    \item Theorem~\ref{tm-existence-of-EE} guarantees the existence and uniqueness of the endemic equilibrium ${\bf E^*} = (S_0^*, S_1^*, S_2^*, Z^*, I^*)^T = (0.7954, 2.1648, 15.0162, 0.6324, 0.4322)^T$, where $S_0^*, S_1^*, S_2^*,Z^*$ are given by~\eqref{equilibrium-def} with $\tilde{I} = I^*$ and $I^*$ is the unique positive solution of Equation~\eqref{I-star}.
    \item Theorem~\ref{prop-2} $(i)$ implies that $I^*$ is non-increasing  with respect to $\tau=\frac{q}{\eta}$ and 
  \[
     \frac{B(\beta_{0,u}-\mathcal{D}_0)}{(\mu+d)(\beta_{0,u}+\frac{q\gamma}{\eta})} = 0.0368 \leq I^*  \le \min\{17.2898, 3.3799\} =  \min\left\{\frac{B(\beta_{0,u}-\mathcal{D}_0)}{\mu\beta_{0,u}},\frac{\mu}{\mathcal{D}_0}\right\}.
  \]
  \item Theorem~\ref{prop-2} $(iii)$ implies that
 \[
     \lim_{\frac{q}{\eta}\to\infty}I^*=I^*_{\infty} = 0.2276,
 \]
 where $I^*_\infty$ is the unique positive solution of the equation $ 
 \mathcal{D}_0=\frac{\mu}{\gamma}\sum_{j=1}^2\frac{\beta_{j,u}\gamma_j}{\mu+\beta_{j,u}I^*_\infty}$. 
 \item Theorem~\ref{prop-3} $(iii)$ implies that 
 \[
\lim_{\frac{q}{\eta}\to\infty}(S_0^*,S_1^*,S_2^*)=\left(0,\frac{\gamma_1B}{\gamma(\mu+\beta_{1,u} I^*_{\infty})},\frac{\gamma_2B}{\gamma(\mu+\beta_{2,u}I^*_\infty)}\right) = (0, 2.2968, 16.9882).
 \]
 
    \item The hypotheses of Theorem~\ref{tm-uniform-persistency} are fulfilled and inequality~\eqref{persistent-eq-1} is satisfied and becomes
    \[
    \frac{\mu(\beta_{0,6} - \mathcal{D}_0)}{\mathcal{D}_0\left(\beta_{0,6} + \frac{q}{\eta}\gamma\right)} = 0.0368\le\limsup_{t\to\infty}I(t)\le 17.2898= \frac{(\mu + d)(\beta_{0,6} - \mathcal{D}_0)}{\mathcal{D}_0\beta_{0,6}},
    \]
    confirming one more time that, the disease will be permanent within this population.
\end{itemize}

\section{Discussion/Conclusion} 
 A close look at the results established in this paper suggests several directions to consider in the future. First, note that Theorem~\ref{Th-5} provides the linear stability of the endemic equilibrium for the model without delay under the condition stated there. However, it will be interesting to establish this stability for the model with delay. In particular, constructing a Lyapunov function for the endemic equilibrium solution to establish global stability would be valuable.
 
 Recalling that the time delay $u>0$ accounts for the intracellular time delay between the initial infection of cells and when a cell becomes infectious, it is natural to suppose that $\beta_{0,u}$ decreases in $u$ and that $\mathlarger{\lim_{u\to\infty}\beta_{0,u}}=0$. In addition, due to practical application of the model, we can also suppose that $\beta_{0,u}>\max\{\beta_{1,u},\beta_{2,u}\}$ for every $u\ge 0$. Under such hypotheses it then becomes natural to take the time delay $u$ as the extra bifurcation parameter in the model. In this case, by Theorem \ref{tm-stability-of-DFE}, we note that the disease-free equilibrium  ${\bf E}^0$ is the global attractor for solutions of \eqref{e-0-1}  for every $u\ge 0$ if $\beta_{0,0}\leq \mathcal{D}_0$. Now, let us suppose that $\beta_{0,0}>\mathcal{D}_0$. Then there is a unique $u^*>0$ such that $\beta_{0,u}>\mathcal{D}_0$ for every $u\in[0,u^*)$ and $\beta_{0,u}\leq \mathcal{D}_0$ for every $u\ge u^*$. As a consequence of Theorem \ref{tm-stability-of-DFE} we conclude that ${\bf E}^0$ is the global attractor for solutions of \eqref{e-0-1} for every $u\ge u^*$. Moreover by Theorem \ref{tm-existence-of-EE}, for every $u\in[0,u^*)$, \eqref{e-0-1} has a unique endemic equilibrium ${\bf E}^*(u)$. As a result, $u=u^*$ is a bifurcation point. It is then interesting to know whether the endemic equilibrium ${\bf E}^*(u)$ is stable for every $u\in[0,u^*)$. 

Another important aspect to consider is the notion of spreading speeds of the disease. One way to study this problem  is through the existence of  traveling  wave solutions with  a certain speed $c$ and connecting the disease-free equilibrium and the endemic equilibrium solutions. We plan to study this question and how the time delay affects the spreading speeds in our future work.

\section*{Acknowledgment} 
The authors are grateful to the anonymous referees for the valuable comments and suggestions.

\bibliographystyle{plain}
\bibliography{ref}
\end{document}